\documentclass[12pt]{article}
\usepackage{amsmath,amssymb,amsthm,amsfonts,mathrsfs,bm}
\usepackage{graphicx,psfrag,epsf,enumerate}
\usepackage{latexsym,multirow,setspace}
\usepackage{authblk,url,xcolor}

\usepackage{hyperref}

\usepackage[utf8]{inputenc}

\usepackage{apacite}

\usepackage{float}
\restylefloat{table}

\DeclareMathOperator*{\argmin}{arg\,min}

\newtheorem{definition}{Definition}[section]
\newtheorem{theorem}[definition]{Theorem}
\newtheorem{lemma}[definition]{Lemma}

\makeatletter
\def\blfootnote{\gdef\@thefnmark{}\@footnotetext}
\makeatother

\begin{document}

\blfootnote{\textup{2000} \textit{ AMS Subject Classification}: 62G08}
\def\spacingset#1{\renewcommand{\baselinestretch}%
{#1}\small\normalsize} \spacingset{1}


\title{The Trimmed Mean in Non-parametric Regression Function Estimation}
\author[1]{Subhra Sankar Dhar}
\author[2]{Prashant Jha}
\author[3]{Prabrisha Rakhshit}
\affil[1]{Department of Mathematics and Statistics, IIT Kanpur, India}
\affil[2]{Department of Mathematics and Statistics, IIT Kanpur, India}
\affil[3]{Department of Statistics, Rutgers University, USA}

\maketitle

\bigskip
\begin{abstract}
This article studies a trimmed version of the Nadaraya-Watson estimator to estimate the unknown non-parametric regression function. The characterization of the estimator through minimization problem is established, and its pointwise asymptotic distribution is derived. The robustness property of the proposed estimator is also studied through breakdown point. Moreover, as the trimmed mean in the location model, here also for a wide range of trimming proportion, the proposed estimator poses good efficiency and high breakdown point for various cases, which is out of the ordinary property for any estimator. Furthermore, the usefulness of the proposed estimator is shown for three benchmark real data and various simulated data.

\end{abstract}

\noindent%

\spacingset{1.45}


\section{Introduction}

We have a random sample $(X_1, Y_1),\ldots ,(X_n, Y_n)$, which are i.i.d.\ copies of $(X, Y)$, and the regression of $Y$ on $X$ is defined as 
\begin{equation}
\label{model}
Y_i = g (X_i) + e_i, \quad i = 1, \ldots, n,
\end{equation} 
where $g(\cdot)$ is unknown, and $e_i$ are independent copies of error random variable $e$ with $E(e \, |\, X)=0$ and $var(e \, |\, X=x) = \sigma^2 (x) < \infty$ for all $x$. Note that the condition $E(e \, |\, X)=0$ is essentially the identifiable condition for mean regression, and it varies over the different procedures of regression. For instance, in the case of the median regression, the identifiable condition will be the median functional $(e | X) = 0$ or for the trimmed mean regression, it will be the trimmed mean functional $(e | X) = 0$.

There have been several attempts to estimate the unknown non-parametric regression function $g (\cdot)$; for overall exposure on this topic, the readers are referred to \citeA{priestley1972non}, \citeA{clark1977non} and \citeA{gasser1979kernel}. Among well-known estimators of the regression function, the Nadaarya-Watson estimator (see \citeA{nadaraya1965non} and \citeA{watson1964smooth} ) is one of the most classical estimator, and it has been used in different Statistical methodologies. Some other classical estimators of the regression function, namely, Gasser-M\"uller \cite{gasser1979kernel} and Priestley-Chao \cite{priestley1972non} estimators are also well-known in the literature. In this context, we should mention that all three aforesaid estimators are based on kernel function; in other words, these estimators are examples of kernel smoothing of regression function. In fact more generally speaking, one may consider local polynomial fitting as a kernel regression smoother, and the fact is that Nadaraya-Watson estimator is nothing but a local constant kernel smoother.  

Note that as it is mentioned in the previous paragraph, Nadaraya-Watson estimator can be obtained from a certain minimization problem related to the weighted least squares methodology, where the weights are the functional values of the kernel function evaluated at data points. This fact further indicates that it is likely to be less efficient in the presence of the outliers or influential observations in the data. To overcome this problem, we here propose the trimmed version of Nadaraya-Watson estimator, which can be obtained as the minimizer of a certain minimization problem. It is also of interest to see how this estimator performs compared to the classical Nadaraya-Watson estimator (i.e., based on the usual least squares methodology) when data follow various distributions. For instance, it should be mentioned that for the location parameter of Cauchy distribution, neither the sample mean nor the sample median, 0.38-trimmed mean is the most efficient estimator for the location parameter of Cauchy distribution as pointed out by \citeA{dhar2016trimmed}. In fact, such a nice combination of efficiency and robustness properties of the trimmed mean for various Statistical model (see, e.g., \citeA{dhar2009comparison,dhar2012derivatives}, \citeA{dhar2016trimmed}, \citeA{vcivzek2016generalized}, \citeA{park2015robust}, \citeA{wang2019robust} and references therein) motivated us to propose this estimator and study its behaviour. For the classical references of the trimmed mean, one may look at \citeA{bickel1965some}, \citeA{hogg1967some}, \citeA{jaeckel1971some}, \citeA{stigler1973asymptotic}, \citeA{welsh1987trimmed}, \citeA{jureckova1994regression} and \citeA{jurevckova1994adaptive}.

The contribution of this article is three fold. The first fold is to propose an entirely new estimator of the non-parametric regression function, which was never studied in the literature before. The next fold is the derivation of the asymptotic distribution of the proposed estimator. As the proposed estimator is based on the order statistic, one cannot use the classical central limit theorem directly; it requires advanced technicalities associated with order statistic to obtain the asymptotic distribution. The last fold is the formal study of the robustness property of the proposed estimator using the concept of breakdown point.

As said before, one of the main crux of the problem is to show the proposed estimator as a minimizer of a certain minimization problem, and that enables us to explain the geometric feature of the estimator. Besides, another difficulty involved in deriving the asymptotic distribution is dealing the order statistics in the non-parametric regression set up. For this reason, one cannot use the classical central limit theorem to establish the asymptotic normality of the estimator after appropriate normalization. Moreover, the presence of kernel function in the expression of the estimator also made challenging to establish the breakdown point of the estimator.

The rest of this article is arranged as follows. Section 2 proposes the estimator and shows how it obtains from the minimization problem. Section 3 provides the large sample properties of the proposed estimator, and the robustness property of the estimator is studied in Section 4. Section 5 presents the finite sample study, and the performance of the estimator for a few benchmark data set is shown in Section 6.  Section 7 contains a few concluding remarks. The proof of Theorem \ref{T1} along with the related lemmas is provided in Appendix A, and Appendix B contains all results of numerical studies in tabular form.


\section{Proposed Estimator}	

Let $(X_{1}, Y_{1}), \ldots, (X_n, Y_n)$ be an i.i.d.\ sequence of random variables having the same joint distribution of $(X, Y)$ and recall the model (\ref{model}). The well-known Nadaraya-Watson estimator is defined as $$\hat{g}_{n, NW} (x_{0}) = \frac{\sum\limits_{i = 1}^{n} k_{n}\left(X_{i} - x_{0}\right) Y_{i}}{\sum\limits_{i = 1}^{n} k_{n}\left(X_{i} - x_{0}\right)},$$ where $\displaystyle k_n (\cdot) = \frac{1}{h_n} K\left(\frac{\cdot}{h_n}\right)$, and $K(\cdot)$ is a symmetric kernel function, i.e., $K$ is a non negative kernel with support $[-\tau ,\tau]$, $\displaystyle \int\limits_{-\tau}^{\tau} K(u) du = 1$ and $K(-u) = K(u)$ for all $u\in [-\tau, \tau]$.  Besides, $\{h_{n}\}$ is a sequence of bandwidth, such that $h_n \rightarrow 0$ as $n\rightarrow \infty$ and $nh_n \rightarrow \infty $ as $n\rightarrow \infty$. Note that $\hat{g}_{n, NW} (x_{0})$ can be expressed as a solution of the following minimization problem :  

\begin{equation}
\label{argmin}
\hat{g}_{n, NW} (x_{0}) = \argmin\limits_\theta \sum\limits_{i=1}^{n} \left(Y_i - \theta \right)^2 k_n \left( X_i - x_{0}\right).
\end{equation}

Note that the above formulation implies that Nadaraya-Watson estimator is a certain weighted average estimator, which can be obtained by weighted least squares methodology. It is a well-known fact that the (weighted) least squares methodology is not robust against the outliers or influential observations (see, e.g., \citeA{huber1981robust}), and to overcome this problem related to the robustness against the outliers, we here study the trimmed version of the weighted least squares methodology, which obtains the local constant trimmed estimator of the non-parametric regression function, i.e., the trimmed version of Nadaraya-Watson estimator. Let us now define the estimator formally, which is denoted by $\hat{g}_{n,\alpha} (\cdot)$ for $\alpha\in [0, \frac{1}{2})$.

\begin{equation}
\label{trimargmin}
\hat{g}_{n,\alpha} (x_{0})=\argmin\limits_\theta \sum\limits_{i = [n \alpha] +1}^{n -[n\alpha ]} \left( Y_{[i]} -\theta \right)^2 k_n \left( X_{(i)} -x_{0} \right),
\end{equation}
where $X_{(i)}$ is the $i$-th ordered observation on variable $X$, and $Y_{[i]}$ denotes the observation on variable $Y$ corresponding $X_{(i)}$. Solving \eqref{trimargmin}, we have 
\begin{equation}
\label{trimnwe}
\hat{g}_{n,\alpha} (x_{0}) = \frac{\sum\limits_{i = [n\alpha] +1}^{n -[n\alpha] } k_n \left( X_{(i)} -x_{0}\right) Y_{[i]} }{ \sum\limits_{i = [n\alpha] +1}^{n -[n\alpha]} k_n \left(X_{(i)} -x_{0}\right)}. 
\end{equation}
Here it should be mentioned that $\alpha\in [0, \frac{1}{2})$ is the trimming proportion, and in particular, for $\alpha = 0$, $\hat{g}_{n,\alpha} (\cdot)$ will coincide with  $\hat{g}_{n, NW} (\cdot)$. i.e., usual Nadaraya-Watson estimator. Further, note that here $Y_{[i]} = g(X_{(i)}) + e_{[i]}$ for $i = 1, \ldots, n$, where $e_{[i]}$ denotes the error corresponding to $(X_{(i)},Y_{[i]})$. 

We now want to add one discussion on possible extension of this estimator. As it follows from (\ref{argmin}), Nadaraya-Watson estimator is a local constant estimator, and in this context, it should be added that there has been an extensive literature on local linear (strictly speaking, local polynomial) estimator of non-parametric regression function. One of the advantage of local linear or generally speaking local polynomial estimator is, it gives a consistent estimator of a certain order derivatives of the regression function as well (see, e.g., \citeA{fan1996local}) but on the other hand, it will enhance the overall standard error as well. Following the same spirit, one can consider local linear or polynomial trimmed mean  of non-parametric regression function, and it will be an interest of future research. 

This section ends with another discussion on the choice of the tuning parameter $\alpha$ involved in $\hat{g}_{n,\alpha} (x_{0})$. Apparently, our efficiency study (see in Section \ref{AE}) and simulation study (see in Section \ref{FSS}) indicate that for a wide range of $\alpha$, $\hat{g}_{n,\alpha} (x_{0})$ has good efficiency property for various distributions. In contrast, in terms of the robustness against the outliers, $\hat{g}_{n,\alpha} (x_{0})$ attains the highest breakdown point when $\alpha$ attains its largest value (see Section \ref{breakdownpoint}). However, since there is a trade-off between efficiency and robustness of an estimator, choosing the largest value of $\alpha$ in practice may originate an estimator having poor efficiency. In order to maintain the best efficiency and a reasonably good breakdown point, one may estimate the trimming proportion $\alpha$ (denote it as $\hat{\alpha}$), which minimizes the estimated asymptotic variance of $\hat{g}_{n,\alpha} (x_{0})$ after appropriate normalization. However, Statistical methodology based on $\hat{g}_{n,\hat{\alpha}} (x_{0})$ will be difficult to implement because of its intractable nature. Overall, the choice of $\alpha$ is an issue of concern to use $\hat{g}_{n,\alpha} (x_{0})$ in practice.


\section{Asymptotic distribution of $\hat{g}_{n,\alpha} (\cdot)$}

In order to implement any Statistical methodology based on $\hat{g}_{n,\alpha} (\cdot)$, one needs to know the distributional behaviour of $\hat{g}_{n,\alpha} (\cdot)$. However, due to complicated form of $\hat{g}_{n,\alpha} (\cdot)$, it is intractable to derive the exact distribution, which drives us to study the asymptotic distribution of $\hat{g}_{n,\alpha} (\cdot)$. This section describes the pointwise asymptotic distribution of $\hat{g}_{n,\alpha} (\cdot)$ after appropriate normalization. To prove the main result, one needs to assume the following conditions. 

\vspace{0.15in}

\noindent {\bf Assumptions : }

\begin{itemize}

\item[(A1)] The regression function $g(\cdot)$ is a real valued continuously twice differentiable function on a compact set. 

\item[(A2)] The probability density function of the covariate random variable $X$, which is denoted by $f_{X}$, is a bounded function. 

\item[(A3)] The kernel $K(\cdot)$ is a bounded probability density function with support $(-\tau,\tau)$ such that, \\
(a) $\displaystyle \int_{-\tau}^{\tau} K (u) du = 1$, \\
(b)	$\displaystyle \int_{-\tau}^{\tau} u K (u) du =0 $, and \\
(c) $\displaystyle h_{n}^{-1}K\left(h_{n}^{-1}\right) = O(1)$.

\item[(A4)] The sequence of bandwidth $\{ h_n \}$ is such that $h_n = O(n^{-1/3})$.

\item[(A5)] The probability density function of error random variable $e$ is symmetric about $0$ with the following properties: 

\noindent (i) $e_{i}$s are i.i.d.\ random variables.

\noindent (ii) For the location functional $T(F)$, $T(F_{e|X}) = 0$, where $F_{e|X}$ is the conditional distribution of $e$ conditioning on $X$. For instance, in the case of $\alpha$-trimmed mean, $T(F) = \frac{1}{(1 - 2\alpha)}\int x dF(x)$, where $\alpha\in [0, \frac{1}{2})$.

\noindent (iii) $E ({e}^2  |\, X = x) = \sigma^2 (x) < \infty$ for all $x$.

\item[(A6)] There exists a positive $\delta$ such that $E(|e_i|^{2+\delta}) < \infty$ for all $i=1,\ldots ,n$.

\end{itemize}

\begin{theorem}
\label{T1}
Under (A1)-(A6) and for any fixed point $x_{0}$, 
$$\sqrt{nh_n} \left( \hat{g}_{n,\alpha} (x_{0}) - g(x_{0}) - h_n^2 k_2 \left( \frac{g^{\prime \prime} (x_{0})}{2(n-2[n\alpha])} \sum\limits_{i=[n \alpha] +1}^{n - [n \alpha]} f_{X_{(i)}} (x_{0}) + \frac{g^{\prime}(x_{0})}{(n-2[n\alpha])} \sum\limits_{i=[n \alpha] +1}^{n - [n \alpha]} f^{\prime}_{X_{(i)}} (x_{0}) \right) \right)$$ converges weakly to a Gaussian distribution with mean $= 0$ and variance $= V$. Here    
$\displaystyle V = \frac{\sigma^2 (x_{0}) \int \{K(u)\}^2 du}{(1 -2\alpha) t_{\alpha} (x_{0}) }$ ,  $\displaystyle k_2 = \int v^2 K(v) dv$, 
 $\displaystyle t_{\alpha} (x_{0})= \lim_{n\to \infty} \frac{1}{n-2[n\alpha]} \sum\limits_{i=[n \alpha] +1}^{n - [n \alpha]} f_{X_{(i)}} (x_{0})$, $f_{X_{(i)}} (x_{0})$ denotes the probability density function of $X_{(i)}$ ($i$-th order statistic of $X$) at the point $x_{0}$, and $g^{'}$ and $g^{''}$ denote the first and the second derivatives of $g$, respectively.
\end{theorem}

The assertion in Theorem \ref{T1} indicates that the rate of convergence of the estimator $\hat{g}_{n,\alpha} (x_{0})$ after proper  transformation is $\sqrt{n h_{n}}$, which is same as the rate of convergence of the usual Nadaraya-Watson estimator. In fact, as $\alpha = 0$, the asymptotic variance of  $$\sqrt{nh_n} \left( \hat{g}_{n,\alpha} (x_{0}) - g(x_{0}) - h_n^2 k_2 \left( \frac{g^{\prime \prime} (x_{0})}{2(n-2[n\alpha])} \sum\limits_{i=[n \alpha] +1}^{n - [n \alpha]} f_{X_{(i)}} (x_{0}) + \frac{g^{\prime}(x_{0})}{(n-2[n\alpha])} \sum\limits_{i=[n \alpha] +1}^{n - [n \alpha]} f^{\prime}_{X_{(i)}} (x_{0}) \right) \right),$$ i.e., $V$ coincides with the asymptotic variance of the Nadaraya-Watson estimator after proper transformation. In other words, from this study and the definition of $\hat{g}_{n,\alpha} (x_{0})$, it follows that $\hat{g}_{n,\alpha} (x_{0})$ coincides with the Nadaraya-Watson estimator (i.e., $\hat{g}_{n, NW} (x_{0})$) when $\alpha = 0$. Besides, the assertion in Theorem \ref{T1} further indicates that the performance or the efficiency of $\hat{g}_{n,\alpha} (x_{0})$ depends on the choice of $\alpha$ and $x_{0}$, and it motivates us to study the asymptotic efficiency of $\hat{g}_{n,\alpha} (x_0)$ for various choices of $\alpha$ and $x_{0}$ in Section \ref{AE}.


\subsection{Asymptotic Efficiency of $\hat{g}_{n,\alpha} (x_{0})$}
\label{AE}

As mentioned earlier, it is of interest to see the asymptotic efficiency of $\hat{g}_{n,\alpha} (x_{0})$ for various choices of $x_{0}$ and $\alpha$ relative to $\hat{g}_{n, NW} (x_{0})$, i.e., usual Nadarya-Watson estimator. To explore this issue, this section studies the asymptotic efficiency of the proposed $\hat{g}_{n,\alpha} (x_{0})$ relative to $\hat{g}_{n, NW} (x_{0})$. Note that when $\alpha = 0$, it follows from the assertion of Theorem \ref{T1} that  the asymptotic variance of $\hat{g}_{n, NW} (x_{0})$ is given by:
\begin{equation*}
AV(\hat{g}_{n, NW} (x_0)) = \frac{\sigma^2 (x_0)}{f_X (x_0)} \int \{K(u)\}^2 du,
\end{equation*}
where $\sigma^2 (x) = E(e^2 |X = x)$, $K$ is the kernel function, and $f_X$ is the density function of $X$. Next, the statement of Theorem \ref{T1} provides us the expression of the asymptotic variance of $\hat{g}_{n,\alpha} (x_{0})$, which is the following. 
\begin{equation*}
AV(\hat{g}_{n,\alpha} (x_0)) = \frac{\sigma^2 (x_{0})}{(1 -2\alpha) t_{\alpha} (x_{0})}  \int \{K(u)\}^2 du,
\end{equation*}
where $\alpha\in [0, \frac{1}{2})$ is the trimming proportion and $\displaystyle t_{\alpha} (x)= \lim_{n\rightarrow\infty} \frac{1}{n-2[n\alpha]} \sum\limits_{i=[n \alpha] +1}^{n - [n \alpha]} f_{X_{(i)}} (x)$. Using $AV (\hat{g}_{n, NW} (x_{0}))$ and $AV (\hat{g}_{n,\alpha} (x_{0}))$, one can compute the asymptotic efficiency (denoted by AE) of $\hat{g}_{n,\alpha} (x_{0})$ relative to $\hat{g}_{n, NW} (x_{0})$, which is as follows:
\begin{equation}
AE(\hat{g}_{n,\alpha} (x_{0}), \hat{g}_{n, NW} (x_{0})) = \frac{AV (\hat{g}_{n, NW} (x_{0}))}{AV (\hat{g}_{n,\alpha} (x_0))} = \frac{(1 - 2\alpha) t_{\alpha} (x_{0})}{f_{X}(x_{0})}.
\end{equation}

It should be mentioned that $AE (\hat{g}_{n,\alpha} (x_{0}), \hat{g}_{n, NW} (x_{0}))$ does not depend on the form of kernel function and the nature of the error random variable unlike the location model although asymptotic variances of both $\hat{g}_{n,\alpha} (x_{0})$ and $\hat{g}_{n, NW} (x_{0})$ depend on the choice of the kernel functions. Besides, note that for $\alpha = 0$, $AE (\hat{g}_{n,\alpha} (x_{0}), \hat{g}_{n, NW} (x_{0})) = 1$ as $\hat{g}_{n,\alpha} (x_{0})$ coincides with $\hat{g}_{n, NW} (x_{0})$ for any $x_{0}$, since the sum $\displaystyle \frac{1}{n} \sum\limits_{i=1}^{n} f_{X_{(i)}} (x_0) = f_X (x_0)$ for any $n$ and $x_0$. This fact can be obtained by the formulation of the probability density function of the order statistic and the properties of the binomial coefficients (see {\bf Fact A} for details in Appendix A). In Figures \ref{ae_plot1} and \ref{ae_plot2}, we plot the AE of $\hat{g}_{n,\alpha} (x_{0})$ relative to $g_{n, NW} (x_{0})$ when $x_{0} = 0.5$, and the co-variate $X$ follows uniform distribution over $(0, 1)$ and beta distribution with the scale parameter $= 2$ and the shape parameter $= 2$, respectively. In both cases, it is observed that the AE of $\hat{g}_{n,\alpha} (0.5)$ relative to $\hat{g}_{n, NW} (0.5)$ is close to one for a wide range of $\alpha\in [0, \frac{1}{2})$. Here it should be mentioned that $AE (\hat{g}_{n,\alpha} (x), \hat{g}_{n, NW} (x))\leq 1$ for all $x$ and $\alpha\in [0, \frac{1}{2})$, which follows from {\bf Fact B} (see in Appendix A). Overall, this study establishes that for even a large values of $\alpha$, $\hat{g}_{n,\alpha} (\cdot)$ can attain the almost the same efficiency as that of $\hat{g}_{n, NW} (\cdot)$ but with having much better breakdown point, which follows from the assertion in Theorem \ref{BP}.

\begin{figure}[p]
 	\begin{center}
 	
 	\includegraphics[scale=.6]{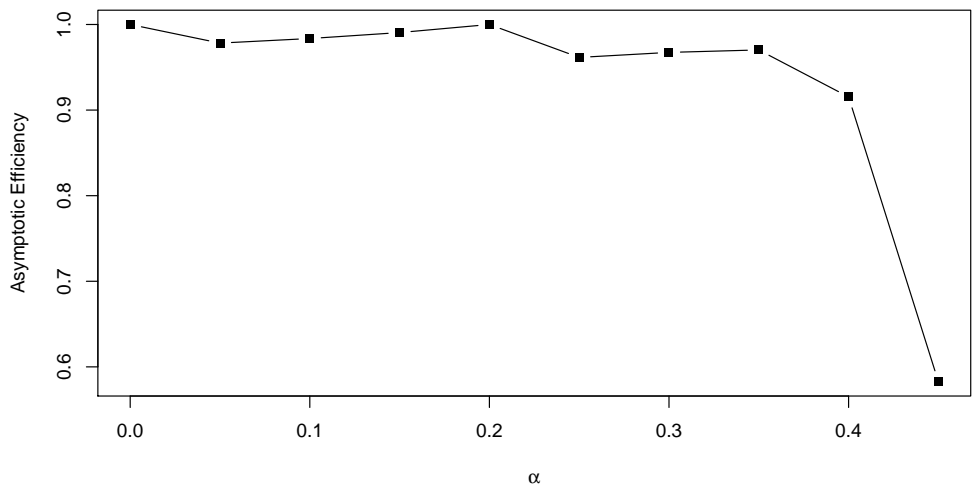}
 
 	\caption{The asymptotic efficiency of $\hat{g}_{n,\alpha} (x_{0})$ relative to $\hat{g}_{n, NW} (x_{0})$ plotted for different values of $\alpha$ when $x_{0} = 0.50$, and $X$ follows uniform distribution over $(0, 1)$. }
 	
 	\label{ae_plot1}
 	\end{center}
\end{figure}

\begin{figure}[p]
 	\begin{center}
 	
 	\includegraphics[scale=.6]{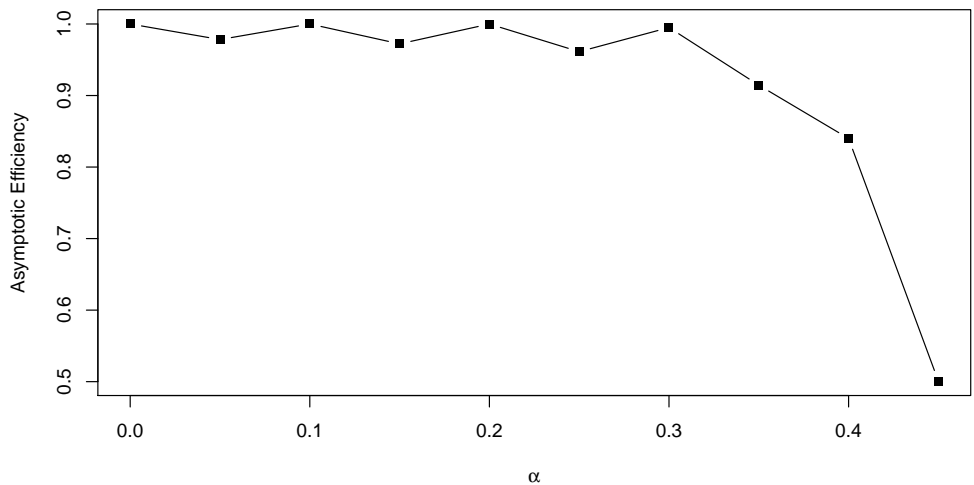}
 
 	\caption{The asymptotic efficiency of $\hat{g}_{n,\alpha} (x_{0})$ relative to $\hat{g}_{n, NW} (x_{0})$ plotted for different values of $\alpha$ when $x_{0} = 0.50$, and $X$ follows beta distribution with the scale parameter $= 2$ and the shape parameter $= 2$.}
 	
 	\label{ae_plot2}
 	\end{center}
\end{figure}


\section{Breakdown Point}
\label{breakdownpoint}
In the earlier section, we established the asymptotic distribution of $\hat{g}_{n,\alpha} (x_{0})$ and studied its asymptotic efficiency for various choices of $\alpha$ and $x_{0}$. Also, it was mentioned that there is a trade-off between the efficiency and the robustness of an estimator. To explore the issue of the robustness of $\hat{g}_{n,\alpha} (x_{0})$, we here study the finite sample breakdown point (see, e.g., \citeA{rousseeuw2005robust}) of $\hat{g}_{n,\alpha} (x_{0})$.

For sake of completeness, we here define the finite sample breakdown point of an estimator $T$. The maximum bias of $T$ with respect to a sample $\mathcal{X} = \{X_1, \ldots, X_n\}$   $(X_{i} \in \mathbb{R}^{d}, d\geqslant 1, i = 1, \ldots, n)$ is defined as $\displaystyle b \left(m, T ,\mathcal{X} \right) = \sup_{\mathcal{X}^{\prime}} |T(\mathcal{X}^{\prime}) - T(\mathcal{X})|$, where $\mathcal{X}^{\prime}$ is the corrupted sample obtained by replacing $m$ sample points from $\mathcal{X}$. Finally, the breakdown point of $T$ is defined as 
\begin{equation*}
\epsilon^{\star}_n \left( T , \mathcal{X} \right) = \inf_{m} \left\{ \frac{m}{n} \, \Big| \, b \left(m, T ,\mathcal{X} \right) \text{ is unbounded} \right\}.
\end{equation*}

\noindent To compute the breakdown point of $\hat{g}_{n,\alpha} (x_{0})$, one needs to assume the following conditions. 

\noindent (B1) The kernel $K(\cdot)$ is a bounded probability density function with support $(-\tau,\tau)$, where $\tau > 0$.

\noindent (B2) $\displaystyle \int_{-\tau}^{\tau} K (u) du = 1$.

\noindent (B3) $\displaystyle h_{n}^{-1}K\left(h_{n}^{-1}\right) = O(1)$.


\noindent The following theorem describes the breakdown point of $\hat{g}_{n,\alpha} (\cdot)$. 

\begin{theorem}
\label{BP}
Under the assumptions (B1), (B2) and (B3), for any $\alpha\in [0, \frac{1}{2})$, the finite sample breakdown point of $\hat{g}_{n,\alpha} (\cdot)$ is $\epsilon^{\star}_n \left( \hat{g}_{n,\alpha} (x_0) , \mathcal{X} \right)$ = $\displaystyle \frac{[n\alpha]}{n}$, and consequently, the asymptotic breakdown point is $\epsilon^{\star} \left( \hat{g}_{n,\alpha} (x_0) , \mathcal{X} \right) = \displaystyle\lim_{n\rightarrow\infty}\epsilon^{\star}_n \left( \hat{g}_{n,\alpha} (x_0), \mathcal{X} \right) = \alpha$.
\end{theorem}

\begin{proof}
Let $\mathcal{X} = \{ (X_1,Y_1), \ldots ,(X_n,Y_n) \}$, then $\hat{g}_{n,\alpha} (x_0)$ based on $\mathcal{X}$ as follows:
\begin{equation*}
\hat{g}_{n,\alpha} (x_0) = \frac{\sum\limits_{i = [n\alpha] +1}^{n -[n\alpha] } k_n \left( X_{(i)} -x_0\right) Y_{[i]} }{ \sum\limits_{i = [n\alpha] +1}^{n -[n\alpha]} k_n \left(X_{(i)} -x_0\right)} = \frac{(n-2[n\alpha])^{-1} \sum\limits_{i = [n\alpha] +1}^{n -[n\alpha] } k_n \left( X_{(i)} -x_0\right) Y_{[i]} }{ (n-2[n\alpha])^{-1} \sum\limits_{i = [n\alpha] +1}^{n -[n\alpha]} k_n \left(X_{(i)} -x_0\right)}.
\end{equation*}
The minimum bias of $\hat{g}_{n,\alpha} (x_0)$ is then expressed as $\displaystyle b \left(m, \hat{g}_{n,\alpha} (x_0) ,\mathcal{X} \right) = \sup_{\mathcal{X}^{\prime}} \left\lVert \hat{g}_{n,\alpha} (x_0)(\mathcal{X}^{\prime}) - \hat{g}_{n,\alpha} (x_0)(\mathcal{X}) \right\rVert $.

Let us first consider $m = [n\alpha]$, and after replacing $m$ pairs of $(X_i,Y_i)$ with arbitrarily large values $(X_i^{\prime},Y_i^{\prime})$, the estimator remains unchanged. The reason is as follows: for all contaminated pairs $(X_i^{\prime},Y_i^{\prime})$, we have $(X_i^{\prime},Y_i^{\prime}) = (X_{(j)},Y_{[j]})$ for some $j > n - [n\alpha]$. This fact implies that $\hat{g}_{n,\alpha} (x)(\mathcal{X}^{\prime}) = \hat{g}_{n,\alpha} (x_0)(\mathcal{X})$ when $[n\alpha]$ number of observations are contaminated. Hence, $\displaystyle b \left([n\alpha], \hat{g}_{n,\alpha} (x_0) ,\mathcal{X} \right) = 0$, and consequently, 
\begin{equation}\label{ub}
\epsilon^{\star}_n \left( \hat{g}_{n,\alpha} (x_0) , \mathcal{X} \right) \geqslant \frac{[n\alpha]}{n}.
\end{equation}
For the reverse inequality, suppose that $m = [n\alpha] +1$ many observations (denoted as $(X_i^{\prime},Y_i^{\prime})$) are corrupted. Note that under this circumstance, for at least one value of $i \in \{1,\ldots ,n\}$, $(X_i^{\prime}, Y_i^{\prime}) = (X_{(j)},Y_{[j]})$ for some $j \leqslant n - [n\alpha]$. Now, the denominator of $\hat{g}_{n,\alpha} (x_0)$ is bounded, i.e.,  $\displaystyle (n-2[n\alpha])^{-1} \sum\limits_{i = [n\alpha] +1}^{n -[n\alpha]} k_n \left(X_{(i)} -x_0\right) < \infty$ for all $x_0$ (using (B3)), and the numerator of $\hat{g}_{n,\alpha} (x_0)(\mathcal{X}^{\prime})$ is unbounded, since the sum $\displaystyle (n-2[n\alpha])^{-1} \sum\limits_{i = [n\alpha] +1}^{n -[n\alpha]} k_n \left(X_{(i)} -x_0\right) Y_{[i]} $ has at least one contaminated $Y_{[i]}$, which makes $\hat{g}_{n,\alpha} (x_0)$ unbounded. Therefore, $b \left([n\alpha] +1, \hat{g}_{n,\alpha} (x_0) ,\mathcal{X} \right)$ becomes unbounded, and hence, 
\begin{equation}\label{lb}
\epsilon^{\star}_n \left( \hat{g}_{n,\alpha} (x_0) , \mathcal{X} \right) < \frac{[n\alpha] +1}{n}.
\end{equation} 

\noindent Combining (\ref{ub}) and (\ref{lb}), we have $\displaystyle\epsilon^{\star}_n \left( \hat{g}_{n,\alpha} (x_0) , \mathcal{X} \right) = \frac{[n\alpha]}{n}$. Consequently, the asymptotic breakdown point is $\epsilon^{\star} \left( \hat{g}_{n,\alpha} (x_0) , \mathcal{X} \right) = \displaystyle\lim_{n\rightarrow\infty}\epsilon^{\star}_n \left( \hat{g}_{n,\alpha} (x_0), \mathcal{X} \right) = \alpha$.
\end{proof}

The assertion in Theorem \ref{BP} indicates that the asymptotic breakdown point of $\hat{g}_{n,\alpha} (x_{0})$ is $\alpha$ for any $x_{0}$, which further implies that it attains the highest asymptotic breakdown point $\frac{1}{2}$ when the trimming proportion $\alpha\rightarrow \frac{1}{2}$. On the other hand, when $\alpha = 0$, the asymptotic breakdown point will be the  lowest possible value zero, which is a formal reason why usual Nadarya-Watson estimator, i.e., $\hat{g}_{n, NW} (.)$ is a non-robust estimator. Overall, the fact is that the robustness of the estimator $\hat{g}_{n,\alpha} (\cdot)$ will increase as the trimming proportion increases unlike the case of efficiency study. In fact, as it is mentioned earlier, this is the reason why the choice of trimming proportion in $\hat{g}_{n,\alpha}$ is an issue of concern as $\alpha$ controls the both the efficiency and the robustness properties of the estimator. Moreover, since the efficiency of the estimator depends on the sample size $n$ as well, we study the finite sample efficiency of the estimator in the next section.


\section{Finite Sample Study}
\label{FSS}
In Section \ref{AE}, we studied the efficiency of $\hat{g}_{n,\alpha} (x_{0})$ for various choices of $\alpha$ and $x_{0}$ when the sample size tends to infinity, and it showed various features in the performance of $\hat{g}_{n,\alpha} (x_{0})$ in terms of the asymptotic efficiency. We are now interested to see the performance of $\hat{g}_{n,\alpha} (x_0)$ for different values of $\alpha$ when the sample size is finite.  In the numerical study, we consider $\alpha = 0.05, 0.10, \ldots ,0.45$ and  $n = 50$ and $500$. Here also, Epanechnikov Kernel (see, e.g., \citeA{silverman1986density}) is used with $h_{n} = \frac{n^{-\frac{1}{2}}}{2}$, and the co-variates are generated from uniform distribution over $(0,1)$, unless mentioned otherwise. We study both linear and non linear model coupled with standard normal and $t$-distribution with 5 degrees of freedom as the distribution of the error random variables. All results are summarized in Figures \ref{fss_plot1}, \ref{fss_plot2}, \ref{fss_plot3} and \ref{fss_plot4}, and in the tabular form in Appendix B.

We compute the finite sample efficiency as follows: Using the form of the model and the distribution of the error random variable, we generate $(x_{1, j}, y_{1, j}), \ldots, (x_{n, j}, y_{n, j})$, when $j = 1, \ldots, N$. Afterwards, for each $j = 1, \ldots, N$, we compute $\hat{g}_{n, \alpha} (x_{0})$ for different values of $\alpha$ and $\hat{g}_{n, NW} (x_{0})$ as well. Let the values of $\hat{g}_{n, \alpha} (x_{0})$ and $\hat{g}_{n, NW} (x_{0})$ be $(U_{1}, \ldots, U_{N})$ and $(V_{1}, \ldots, V_{N})$, respectively, and finally, the finite sample efficiency of $\hat{g}_{n, \alpha} (x_{0})$ relative to $\hat{g}_{n, NW} (x_{0})$ is defined as $\frac{\frac{1}{N}\sum\limits_{i = 1}^{N} \left(V_{i} - \frac{1}{N}\sum\limits_{i = 1}^{N} V_{i} \right)^{2} }{\frac{1}{N} \sum\limits_{i = 1}^{N} \left(U_{i} - \frac{1}{N}\sum\limits_{i = 1}^{N}U_{i}\right)^{2}}$. In the numerical study, we consider $N = 1000$ and $x_{0} = 0.50$. In Figures \ref{fss_plot1}, \ref{fss_plot2}, \ref{fss_plot3} and \ref{fss_plot4}, we plot the finite sample efficiency of $\hat{g}_{n,\alpha}$ relative to $\hat{g}_{n,NW}$ for different values of $\alpha$, for four different underlying model.

\vspace{0.25in}

\noindent {\bf Example 1:} Model: $Y = 5X+e$, where $e$ follows standard normal distribution.

\begin{figure}[H]
	\begin{center}
	
	\includegraphics[scale=.6]{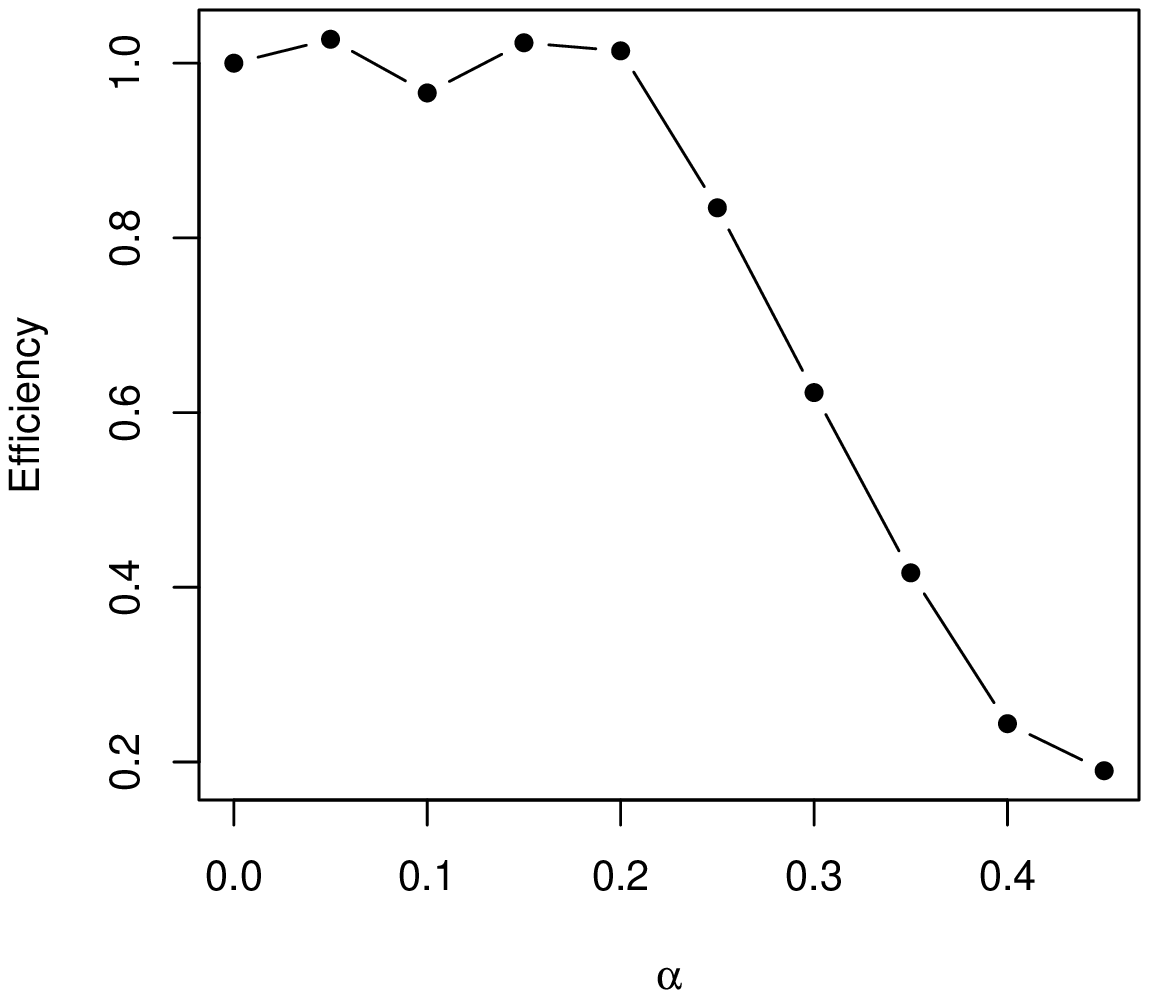}
	\includegraphics[scale=.6]{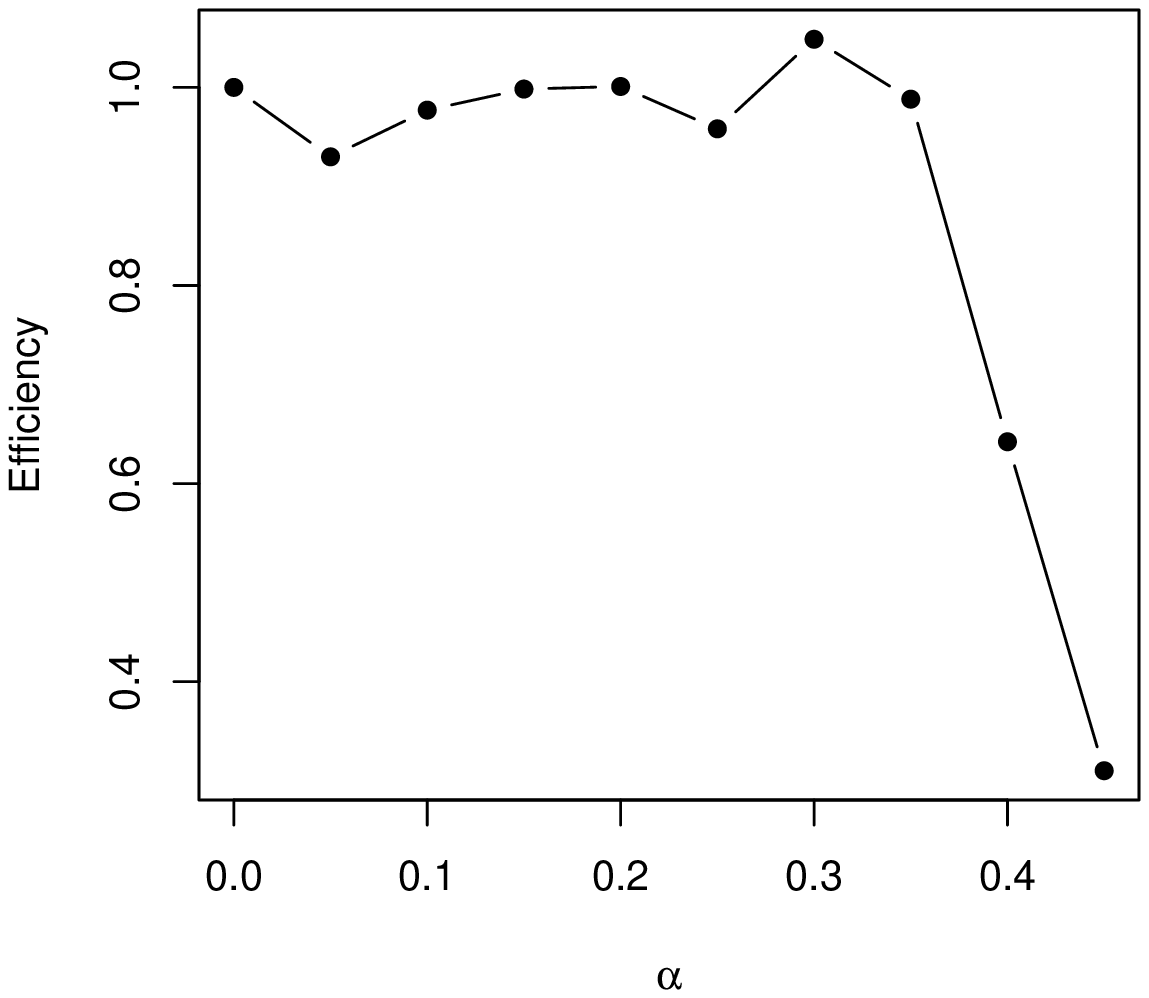}

	\caption{For different values of $\alpha$, the finite sample efficiency of $\hat{g}_{n,\alpha}(0.5)$ relative to $\hat{g}_{n,NW}(0.5)$ for Example 1. The left diagram is for $n = 50$, and the right diagram is for $n = 500$.}
	\label{fss_plot1}
	\end{center}
\end{figure}

The diagrams in Figures \ref{fss_plot1}, \ref{fss_plot2}, \ref{fss_plot3} and \ref{fss_plot4} indicate that $\hat{g}_{n, \alpha} (x_{0})$ has good efficiency relative to $\hat{g}_{n, NW} (x_{0})$ for a wide range of $\alpha$ regardless of the choice of the model and/or the choice of the distribution of the error random variable.

\vspace{0.25in}

\noindent {\bf Example 2:} Model: $Y = 4X^3 +e$, where $e$ follows standard normal distribution.

\begin{figure}[H]
	\begin{center}

	\includegraphics[scale=.6]{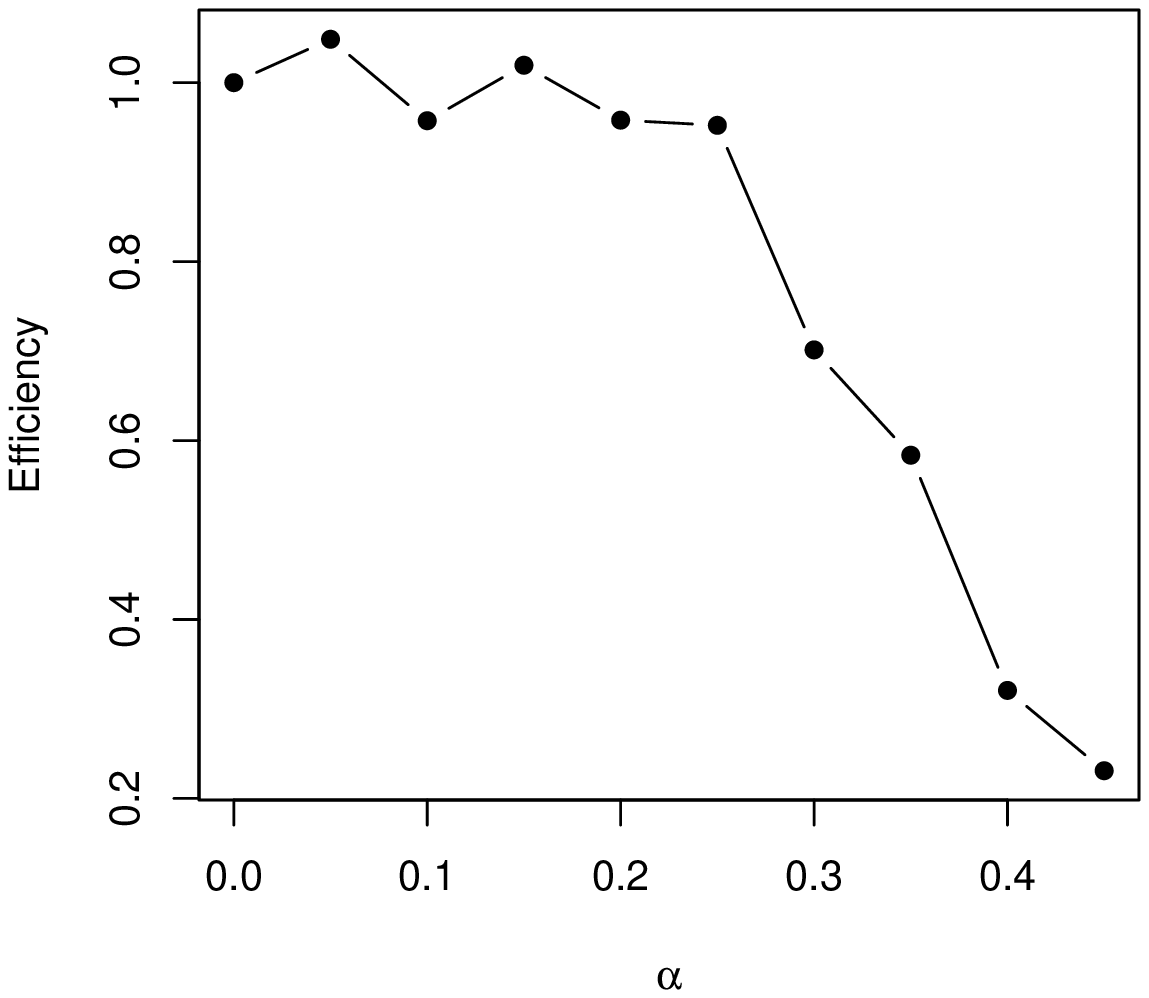}
	\includegraphics[scale=.6]{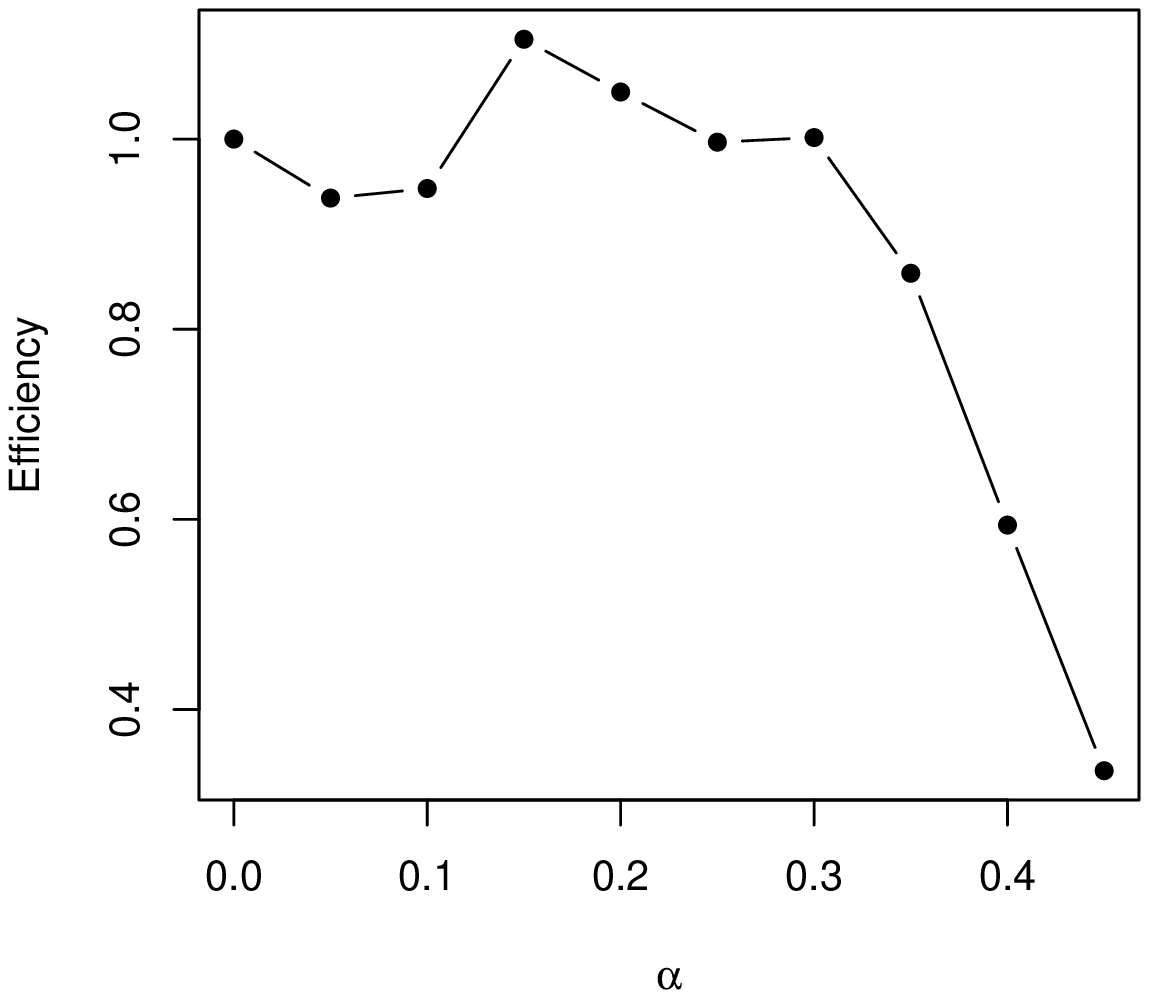}

	\caption{For different values of $\alpha$, the finite sample efficiency of $\hat{g}_{n,\alpha}(0.5)$ relative to $\hat{g}_{n,NW}(0.5)$ for Example 2. The left diagram is for $n = 50$, and the right diagram is for $n = 500$.}
	\label{fss_plot2}
	\end{center}
\end{figure}

\noindent {\bf Example 3:} Model: $Y = 5X +e$, where $e$ follows $t$-distribution with $5$ degrees of freedom.

\begin{figure}[H]
	\begin{center}

	\includegraphics[scale=.6]{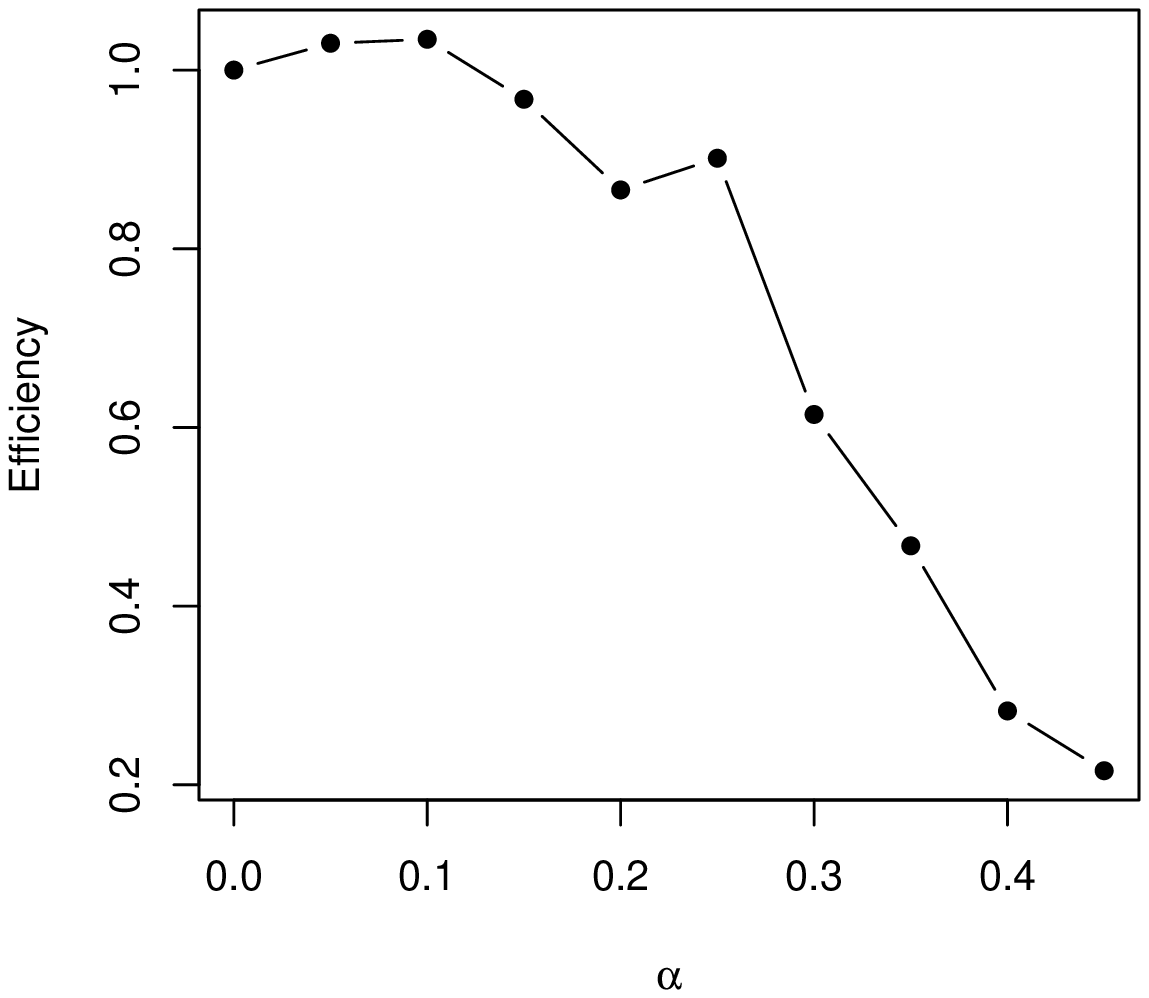}
	\includegraphics[scale=.6]{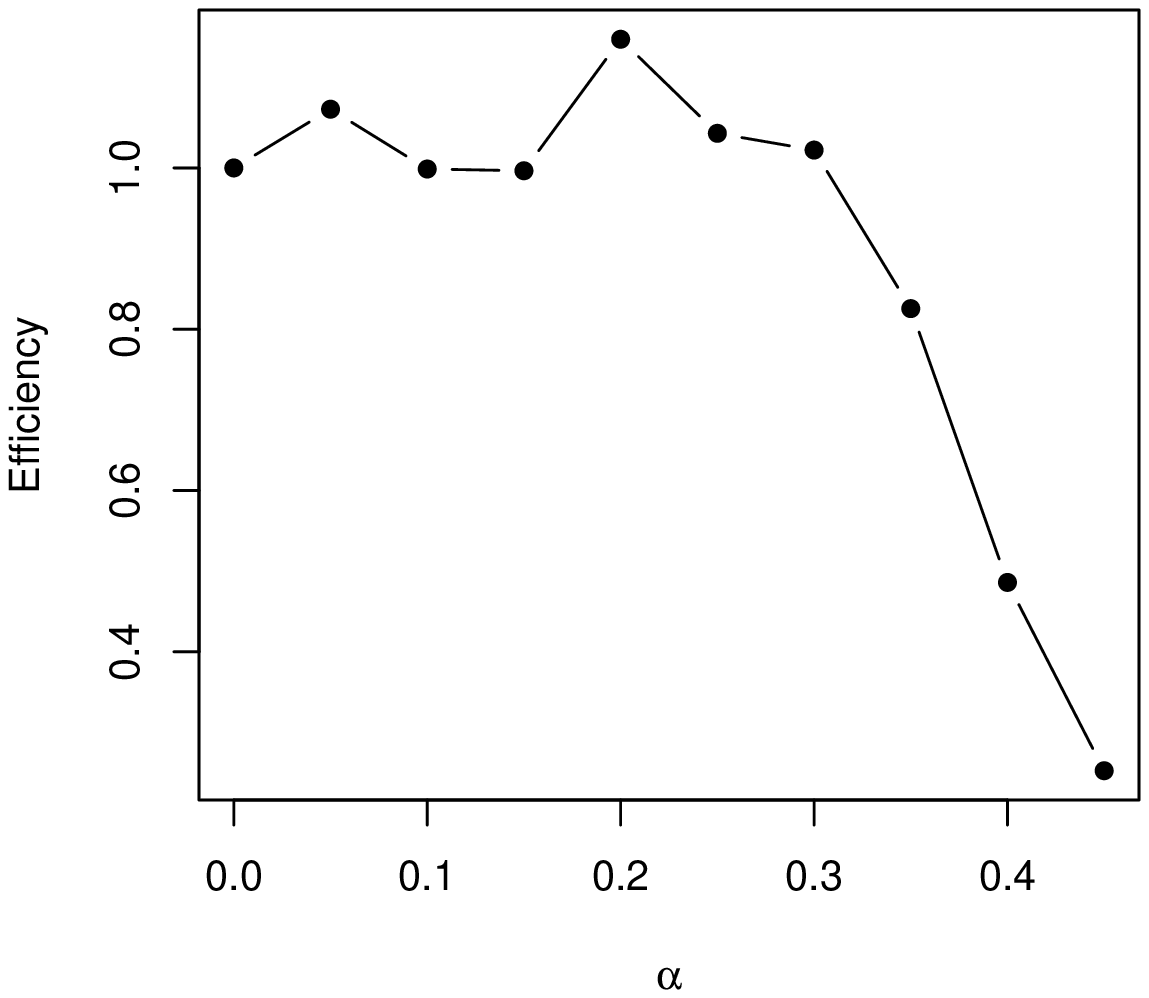}

	\caption{For different values of $\alpha$, the finite sample efficiency of $\hat{g}_{n,\alpha}(0.5)$ relative to $\hat{g}_{n,NW}(0.5)$ for Example 3. The left diagram is for $n = 50$, and the right diagram is for $n = 500$.}
	\label{fss_plot3}
	\end{center}
\end{figure}

\noindent {\bf Example 4:} Model: $Y = 4X^3 +e$, where $e$ follows $t$-distribution with $5$ degrees of freedom.

\begin{figure}[H]
	\begin{center}

	\includegraphics[scale=.6]{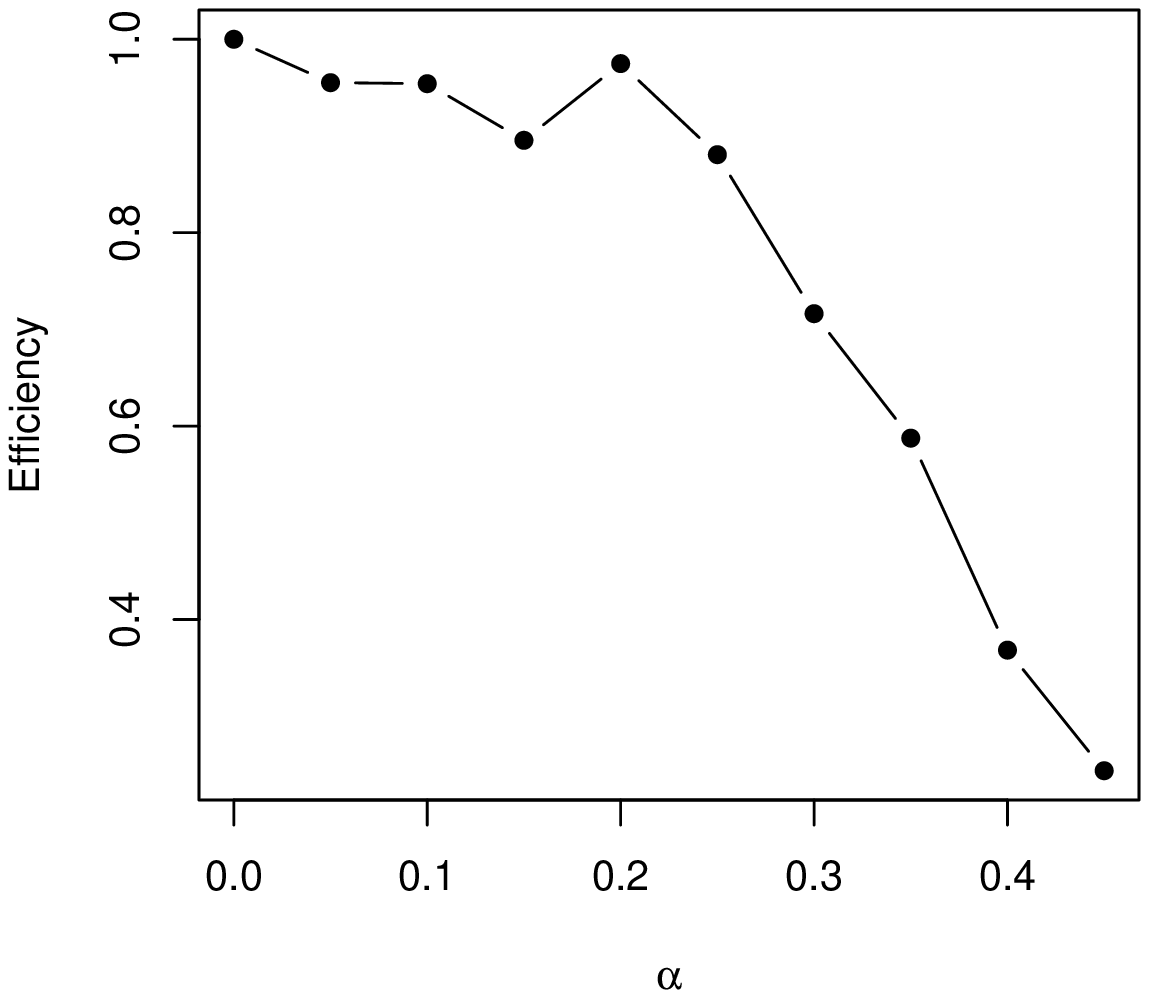}
	\includegraphics[scale=.6]{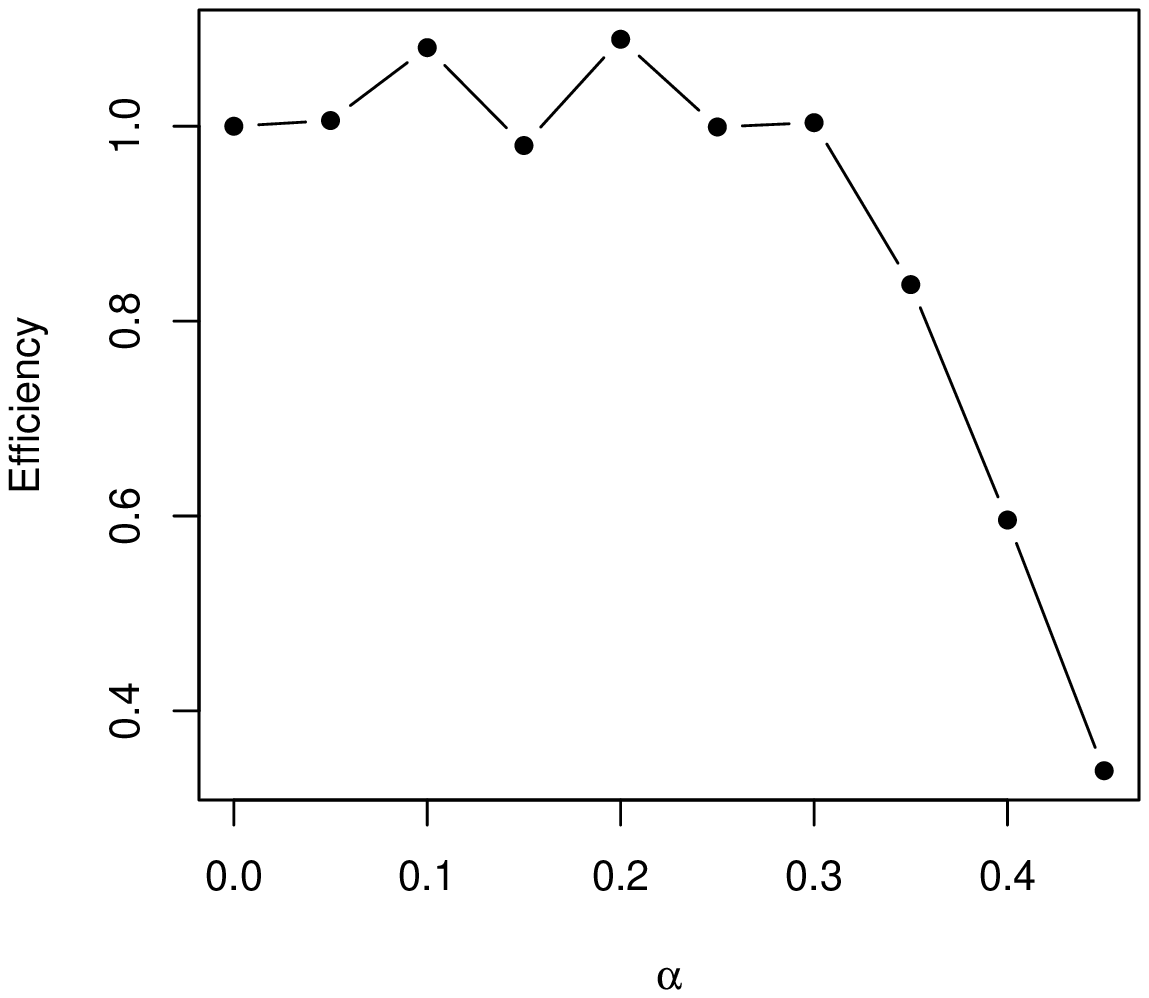}

	\caption{For different values of $\alpha$, the finite sample efficiency of $\hat{g}_{n,\alpha}(0.5)$ relative to $\hat{g}_{n,NW}(0.5)$ for Example 4. The left diagram is for $n = 50$, and the right diagram is for $n = 500$.}
	\label{fss_plot4}
	\end{center}
\end{figure}


\section{Real Data Analysis}
\label{rda}

In this section, we illustrate the functionality of our proposed estimator on some benchmark real data sets. All these data sets are available in UCI machine repository. 

{\bf Combined Cycle Power Plant Data Set:}  This data set contains 9568 data points collected from a Combined Cycle Power Plant over six years (2006-2011), when the plant was set to work with full load (see \citeA{kaya2012local} and \citeA{tufekci2014prediction} for details). The data set can be accessed with the following link:  \url{https://archive.ics.uci.edu/ml/datasets/Combined+Cycle+Power+Plant}. The data contains five attributes: Temperature (in $^{\circ}$C), Ambient Pressure (in milibar), Relative Humidity (in \%), Exhaust Vacuum (in cm Hg) and Electrical Energy Output (in MW). We consider Temperature as our co-variate or independent variable ($X$) and Electrical Energy output as response or dependent variable ($Y$). We provide a scatter plot of the Electrical Energy Output against Temperature in the first diagram of Figure \ref{realdata1}.

In the study, we first scaled the data associated with the co-variate to the interval $[0, 1]$ using the transformation $x^{*} = \frac{x - \min (x)}{\max(x) - \min(x)}$, where $x^{*}$ is the transformed variable, and we adopt Bootstrap methodology to compute the efficiency, which is called as Bootstrap efficiency. The procedure : We first generate $B$ many Bootstrap resamples with size $n$ from the data $(y_{1}, x_{1}^{*}), \ldots, (y_{n}, x_{n}^{*})$, and compute the values of $\hat{g}_{n,\alpha} (x_{0})$ and $\hat{g}_{n, NW} (x_{0})$ for each resample. Let us denote those values of $\hat{g}_{n,\alpha} (x_{0})$ and $\hat{g}_{n, NW} (x_{0})$ as $(S_{1}, \ldots, S_{B})$ and $(R_{1}, \ldots, R_{B})$, respectively. Then the Bootstrap efficiency of $\hat{g}_{n,\alpha} (x_{0})$ relative to $\hat{g}_{n, NW} (x_{0})$ is defined as $\frac{\frac{1}{B}\sum\limits_{i = 1}^{B} \left(R_{i} - \frac{1}{B}\sum\limits_{i = 1}^{B}R_{i}\right)^{2}}{\frac{1}{B}\sum\limits_{i = 1}^{B} \left(S_{i} - \frac{1}{B}\sum\limits_{i = 1}^{B}S_{i}\right)^{2}}.$ In the second diagram of Figure \ref{realdata1}, we plot the Bootstrap efficiency of $\hat{g}_{n,\alpha} (x_{0})$ relative to $\hat{g}_{n, NW} (x_{0})$ for different values of $\alpha$ for this data (here $n = 9568$). In this study, we consider Epanechnikov kernel with bandwidth $h_{n} = \frac{n^{-\frac{1}{2}}}{2}$ and $B = 1000$. The second diagram of Figure \ref{realdata1} indicates that the efficiency of $\hat{g}_{n,\alpha}$ relative to $\hat{g}_{n, NW}$ is substantially high for a wide range of $\alpha$, and the most probable reason is that the data has a few influential observations.

\begin{figure}
	\begin{center}
	
	\includegraphics[scale=.7]{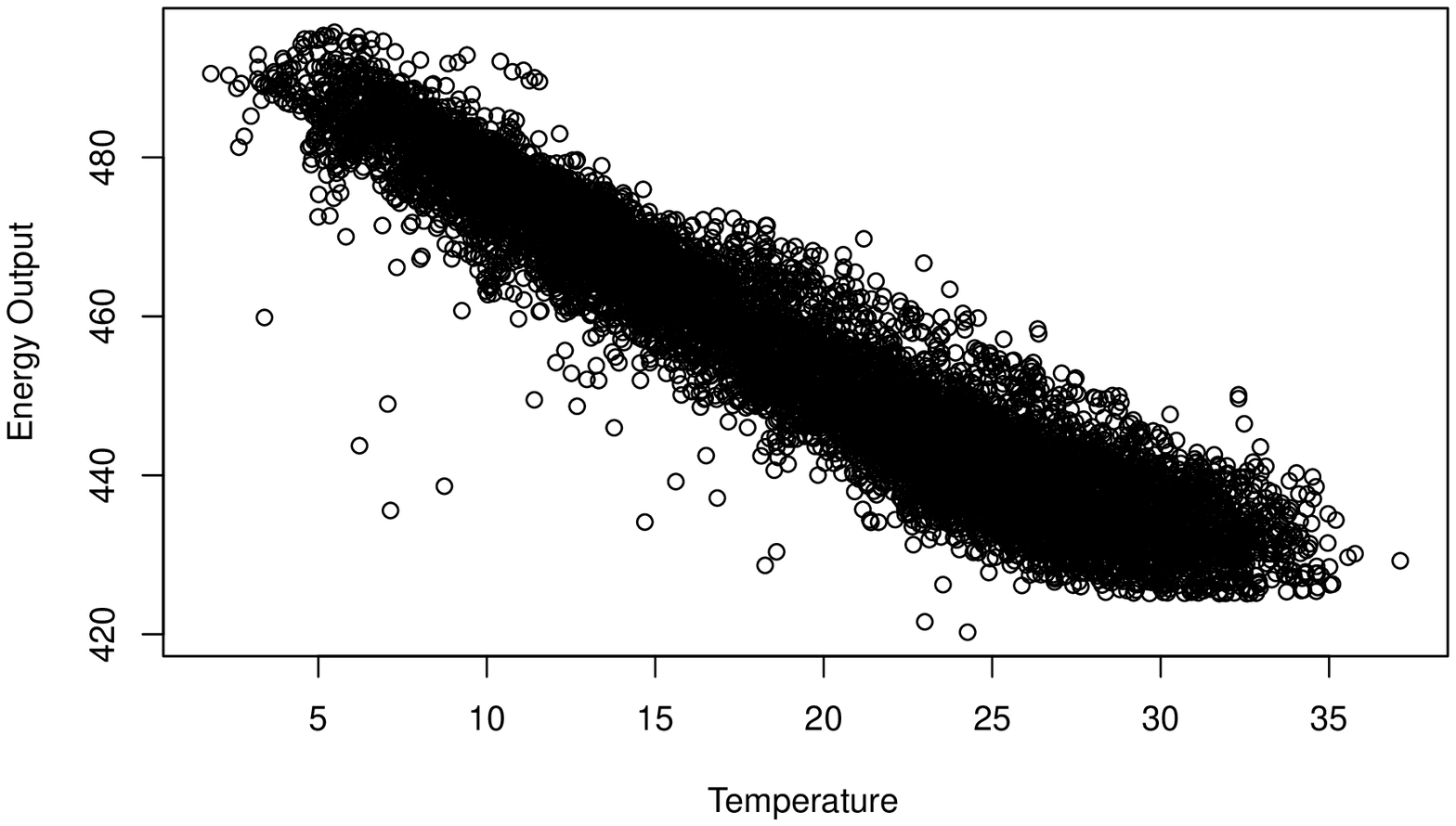}
	\includegraphics[scale=.7]{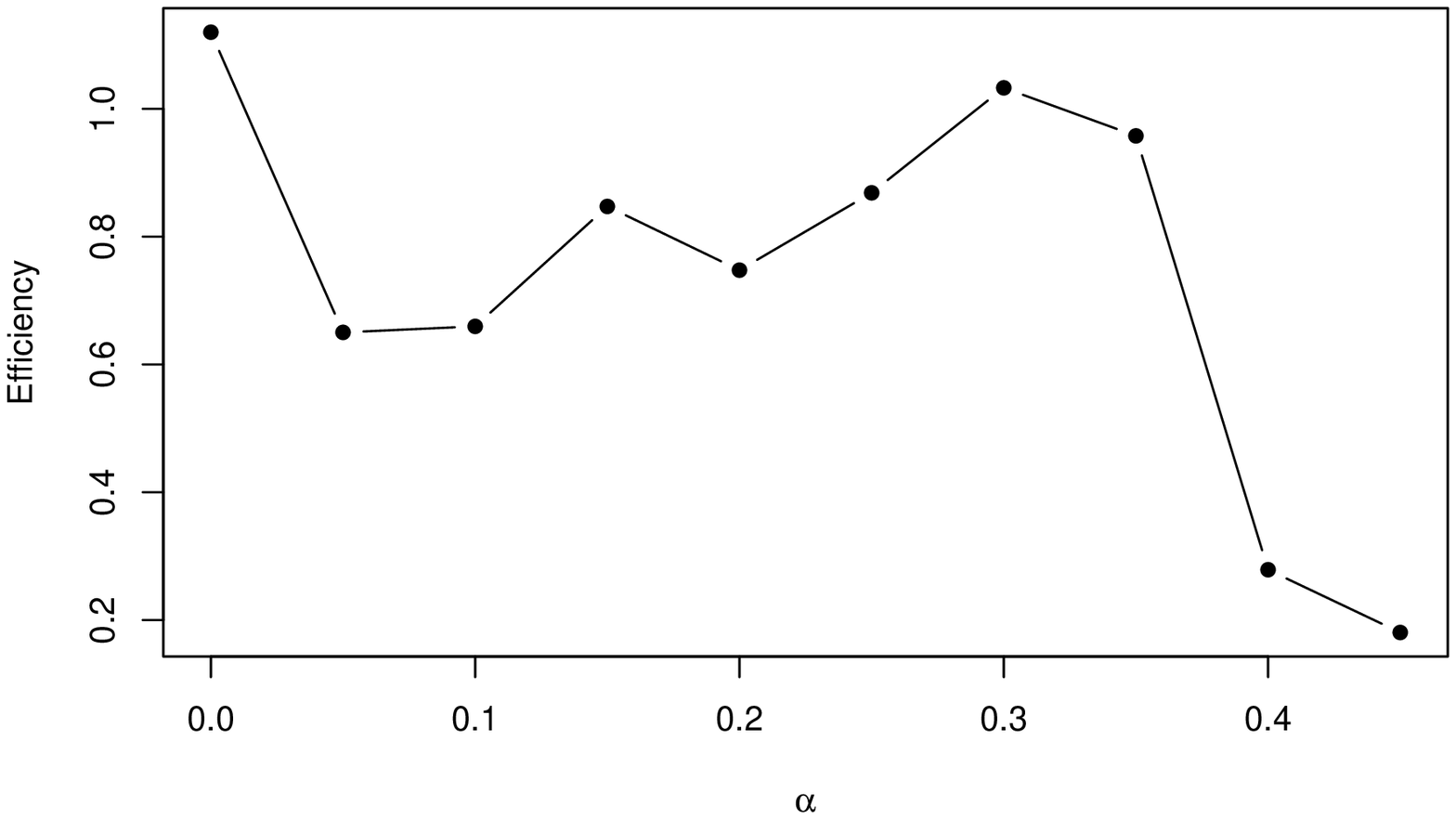}
	
	\caption{Scatter plot (Up) of the Combined Cycle Power Plant Data and Efficiency (Down) of $\hat{g}_{n,\alpha} (x_0)$ relative to $\hat{g}_{n,NW} (x_0)$ at $x_0 =0.50$ plotted for different values of $\alpha$.}
	\label{realdata1}
	\end{center}
\end{figure}

{\bf Parkinson's Telemonitoring Data Set:} This data set consists of 5875 recordings of several medical voice measures from forty two people with early-stage Parkinson's disease recruited to a six month trial of a telemonitoring device, for remote symptom progression monitoring (see \citeA{tsanas2009accurate} for details). This data can be accessed with the following link: \url{https://archive.ics.uci.edu/ml/datasets/Parkinsons+Telemonitoring}. The data has 22 attributes, out of which two attributes are of our interest for this study: NHR (measure of ratio of noise to tonal components in the voice) and RPDE (A nonlinear dynamical complexity measure). We consider the attribute NHR as our co-variate or independent variable ($X$) and RPDE as our response or dependent variable ($Y$). We provide a scatter plot of RPDE against NHR variable in the first diagram of Figure \ref{realdata2}. Unlike the earlier data analysis, the transformation of the co-variate has not been done here as the values of NHR variable belongs to $[0, 1]$.

Here also, we compute the Bootstrap efficiency $\hat{g}_{n,\alpha} (x_{0})$ relative to $\hat{g}_{n, NW} (x_{0})$ based on the data $(Y, X)$, and the procedure is same as it is described in the earlier real data analysis. In the second diagram of Figure \ref{realdata2}, we plot the Bootstrap efficiency of $\hat{g}_{n,\alpha} (x_{0})$ relative to $\hat{g}_{n, NW} (x_{0})$ for different values of $\alpha$ for this data (here $n = 5875$). In this study also, we consider Epanechnikov kernel with bandwidth $h_{n} = \frac{n^{-\frac{1}{2}}}{2}$ and $B = 1000$. Here also, the second diagram of Figure \ref{realdata1} indicates the efficiency of $\hat{g}_{n,\alpha}$ relative to $\hat{g}_{n, NW}$ is substantially high for a wide range of $\alpha$, and as we indicated in the earlier study, the most probable reason is that the data has a few influential observations.

\begin{figure}
	\begin{center}
	
	\includegraphics[scale=.7]{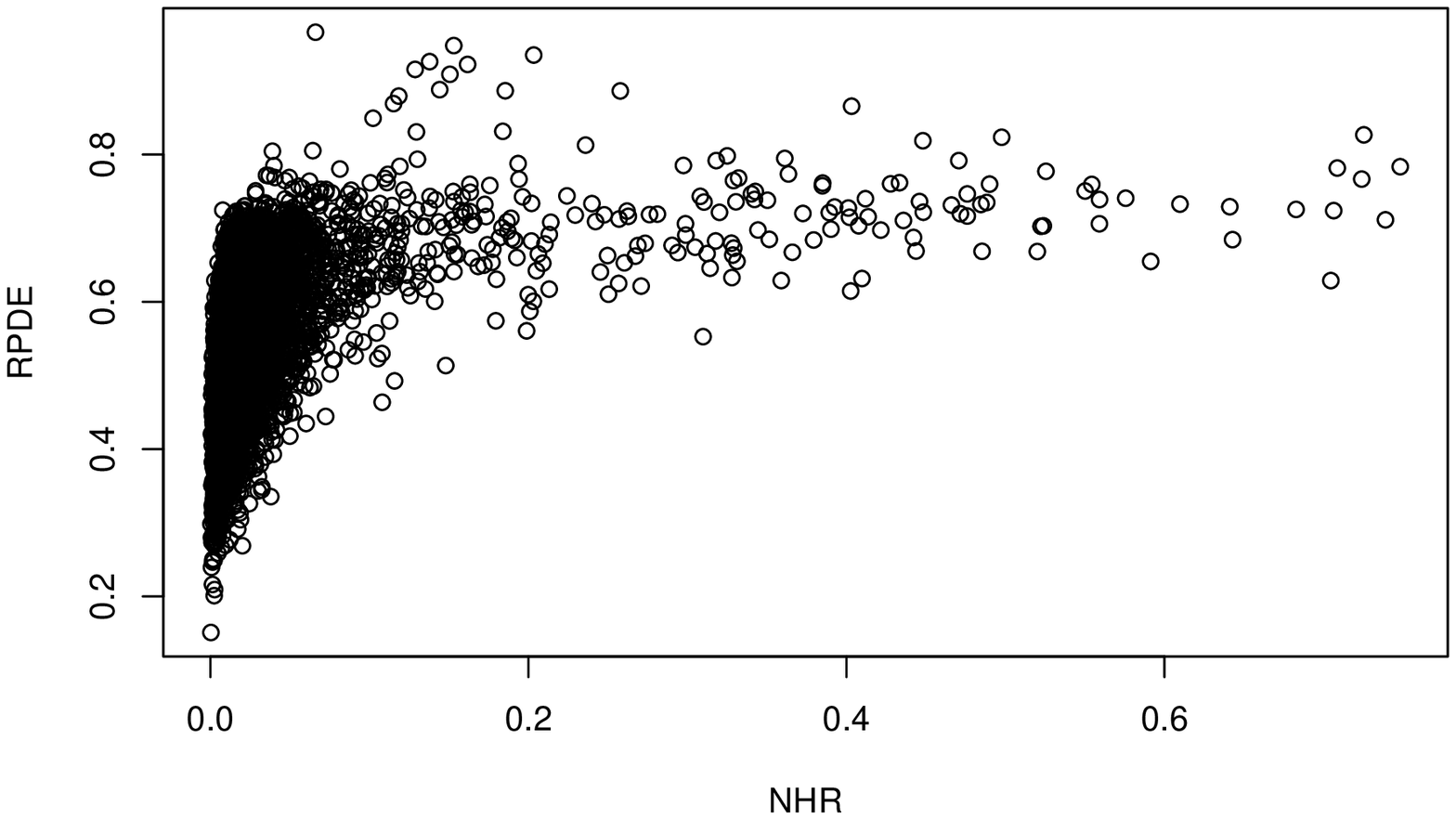}
	\includegraphics[scale=.7]{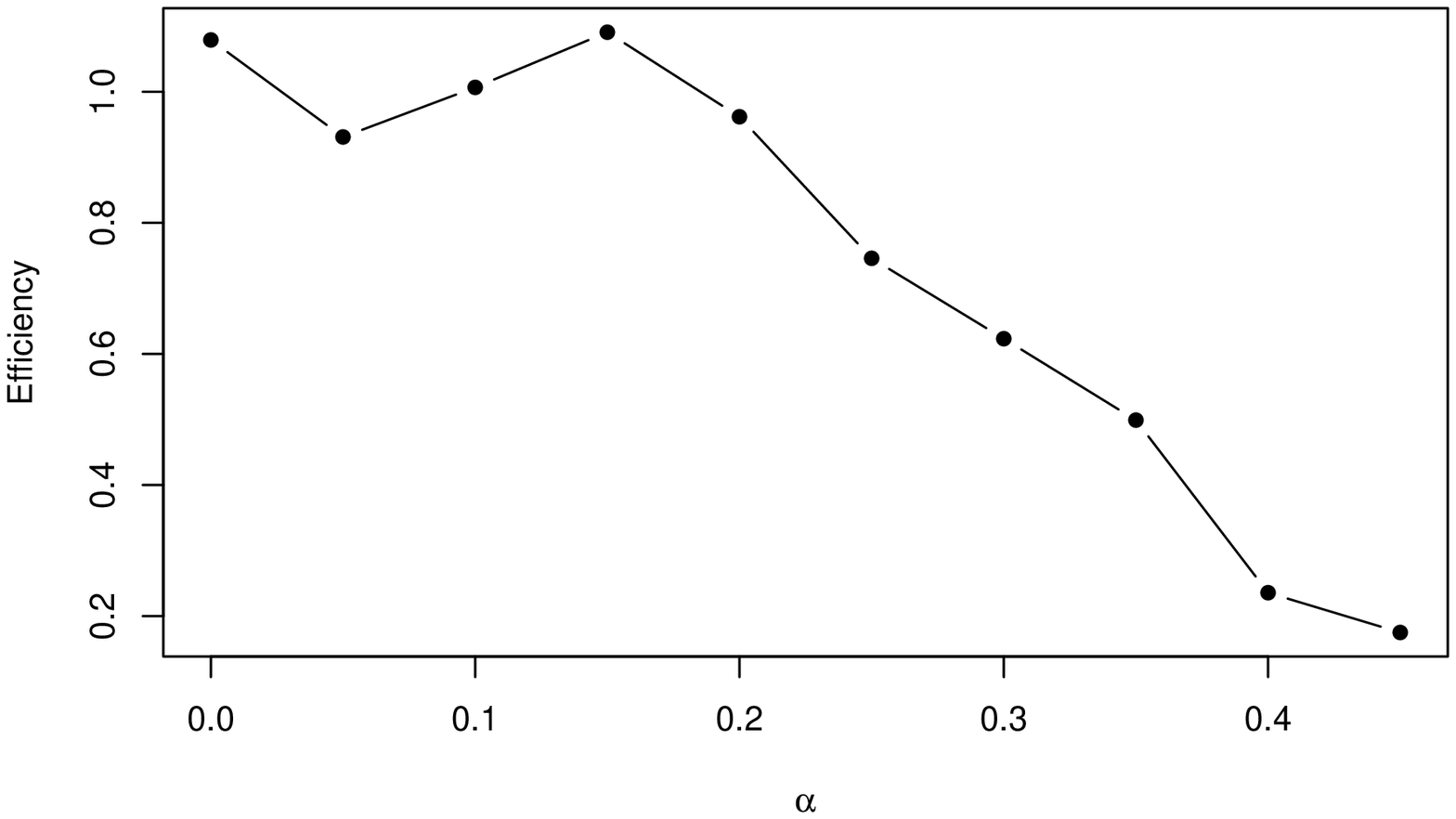}
	
	\caption{Scatter plot (Up) of the Parkinson's Telemonitoring Data and Efficiency (Down) of $\hat{g}_{n,\alpha} (x_0)$ relative to $\hat{g}_{n,NW} (x_0)$ at $x_0 =0.50$ plotted for different values of $\alpha$.}
	\label{realdata2}
	\end{center}
\end{figure}

{\bf Air Quality Data Set:} This data set contains 9358 instances of hourly averaged responses from an array of five metal oxide chemical sensors embedded in an Air Quality Chemical Multisensor Device (see \citeA{de2008field} for details). The device was set up in a polluted area of an Italian city, at road level. The data set can be accessed with the following link: \url{https://archive.ics.uci.edu/ml/datasets/Air+quality}. There are thirteen attributes in this data apart from date and time: \\
1 True hourly averaged concentration $CO$ in $mg/m^3$ (reference analyzer). \\
2 PT08.S1 (tin oxide) hourly averaged sensor response (nominally CO targeted). \\
3 True hourly averaged overall Non Metanic Hydrocarbons concentration in $microg/m^3$ (reference analyzer). \\
4 True hourly averaged Benzene concentration in $microg/m^3$ (reference analyzer). \\
5 PT08.S2 (titania) hourly averaged sensor response (nominally NMHC targeted). \\
6 True hourly averaged NOx concentration in ppb (reference analyzer). \\
7 PT08.S3 (tungsten oxide) hourly averaged sensor response (nominally NOx targeted). \\
8 True hourly averaged NO2 concentration in $microg/m^3$ (reference analyzer). \\
9 PT08.S4 (tungsten oxide) hourly averaged sensor response (nominally NO2 targeted). \\
10 PT08.S5 (indium oxide) hourly averaged sensor response (nominally O3 targeted).\\
11 Temperature in $^{\circ}$C. \\
12 Relative Humidity (\%). \\
13 AH: Absolute Humidity. \\
We consider the quantity of Tungsten Oxide as the co-variate or the independent variable ($X$), and the Absolute Humidity is considered as the response or dependent variable ($Y$). We provide a scatter plot of Absolute Humidity against Tungsten Oxide in the first diagram of Figure \ref{realdata3}, and the scatter plot indicates that the data does not have as such any influential or outlier observations.

In the study, as we did for the first real data analysis,  the co-variate is scaled to the interval $[0, 1]$ using the transformation $x^{*} = \frac{x - \min (x)}{\max(x) - \min(x)}$, where $x^{*}$ is the transformed variable, and we here also adopt Bootstrap methodology to compute the efficiency. The computational procedure of the Bootstrap efficiency is same as we described in the first real data analysis. In the second diagram of Figure \ref{realdata3}, we plot the Bootstrap efficiency of $\hat{g}_{n,\alpha} (x_{0})$ relative to $\hat{g}_{n, NW} (x_{0})$ for different values of $\alpha$ for this data (here $n = 9358$). In this study also, we consider Epanechnikov kernel with bandwidth $h_{n} = \frac{n^{-\frac{1}{2}}}{2}$ and $B = 1000$. The second diagram of Figure \ref{realdata3} indicates the efficiency of $\hat{g}_{n,\alpha}$ relative to $\hat{g}_{n, NW}$ is high for small values $\alpha$, and decreases gradually as $\alpha$ increases. It is expected since the data does not have any substantial outliers, $\hat{g}_{n,\alpha} (x_{0})$ performs well for small values of $\alpha$, i.e., when $\hat{g}_{n,\alpha}$ is almost same as $\hat{g}_{n, NW}$. On the other hand, the efficiency of $\hat{g}_{n,\alpha}$ steadily goes down for larger values of $\alpha$ as $\hat{g}_{n,\alpha}$ is different from $\hat{g}_{n, NW}$ by a great amount when $\alpha$ is large.

\begin{figure}
	\begin{center}
	
	\includegraphics[scale=.7]{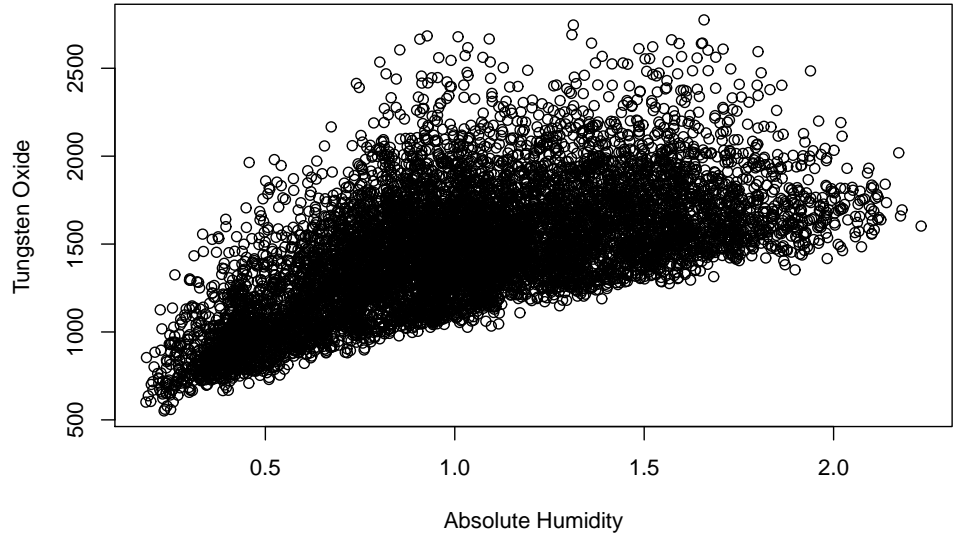}
	\includegraphics[scale=.7]{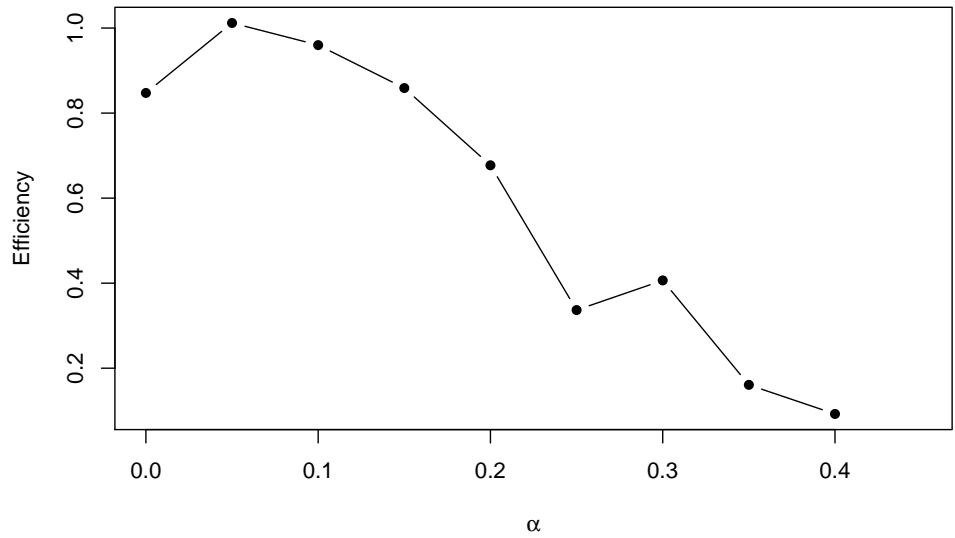}
	
	\caption{Scatter plot (Up) of the Air Quality Data and Efficiency (Down) of $\hat{g}_{n,\alpha} (x_0)$ relative to $\hat{g}_{n,NW} (x_0)$ at $x_0 =0.50$ plotted for different values of $\alpha$.}
	\label{realdata3}
	\end{center}
\end{figure}


\section{Concluding Remarks}
 
{\bf Local linear or local polynomial version:} 

The estimator studied in this article is a local constant trimmed mean for the non-parametric regression function (see (\ref{trimargmin})). Following the same spirit of (\ref{trimargmin}), one can define a local linear or even local polynomial version of the trimmed mean for non-parametric regression function. However, at the same time, adding more variables may create various problems associated with the issue of variable selection. Choosing appropriate degree of polynomial version of the trimmed mean may be an interest of future research.  Besides, one of the well-known problem of using the local constant estimator is the adverse effect of boundary (see e.g., \citeA{fan1996local}). However, since the trimming based estimator is based on the ordered observations and the procedure of trimming, the proposed estimator can avoid the negative effect of boundary. 

\vspace{0.1in}

\noindent{\bf Uniform convergence and influence function:}

In Theorem \ref {T1}, we stated the pointwise weak convergence of $\hat{g}_{n,\alpha} (x)$, and it is indeed true that the result would be more appealing if one can establish the process convergence of $\hat{g}_{n,\alpha}(x)$, which  allows us to study the related testing of hypothesis problem based on $\hat{g}_{n,\alpha} (x)$. However, to prove the process convergence of  $\hat{g}_{n,\alpha} (x)$, one needs to establish the tightness property of $\hat{g}_{n,\alpha} (x)$, which is not easily doable. Regarding the robustness property of $\hat{g}_{n,\alpha} (x)$, along with the breakdown point, one may also consider the gross error sensitivity (see \citeA{huber1981robust}, p.14) as a measure of robustness. However, since the gross error sensitivity only measures the local robustness of an estimator whereas the breakdown point measures the global robustness of the estimator, we here investigate the breakdown point of the proposed estimator. 

\vspace{0.1in}

\noindent{\bf The choice of kernel function and bandwidth:}

The choice of kernel function along with its bandwidth another issue of concern. In our numerical study, we consider Epanechnikov kernel since it is the most efficient kernel among the symmetric kernel (see \citeA{silverman1986density}, p.59). Regarding the choice of bandwidth, we consider $h_{n} = \frac{n^{-\frac{1}{2}}}{2}$ since it satisfies $h_{n}\rightarrow 0$ and $n h_{n}\rightarrow\infty$ as $n\rightarrow\infty$. However, since the aforesaid criterion is asymptotic in nature, one can adopt the methodology based on data driven approach but using such choice of bandwidth, deriving the asymptotic distribution of the proposed estimator will become more challenging. 

\vspace{0.1in}

\noindent{\bf Main contribution of this article:}

In this article, we propose a new estimator for the non-parametric regression function, which coincides with the well-known Nadarya-Watson estimator as a special case. The characterization of the proposed estimator through an optimization problem is also discussed. In the study, we have observed that the proposed estimator can maintain a good efficiency with high break down point for a wide range of trimming proportion, which is a rare attribute of any estimator. The estimator performs well on real data as well.


\section{Appendix}

\subsection{Appendix A : Proofs}

We first present a fact, which is used to compute the asymptotic efficiency of $\hat{g}_{n,\alpha} (x_{0})$ relative to $\hat{g}_{n, NW} (x_{0})$ in Section 3.1. 

\noindent {\bf Fact A:} {\it Let $X_{(i)}$ be the $i$-th order statistic of the i.i.d.\ random variables $\{X_1,\ldots ,X_n\}$ with common density $f_X$ and distribution function $F_X$. Suppose that $f_{X_{(i)}}$ is the probability density function of $X_{(i)}$ for $i=1,\ldots ,n$. Then for any $x_{0}$,  
$$\frac{1}{n}\sum_{i = 1}^{n} f_{X_{(i)}} (x_{0}) = f_{X} (x_{0}).$$}

\noindent {\bf Proof of Fact A:} Note that for any arbitrary $x_0$, and $i=1,\ldots ,n$, we have 
\begin{equation*}
f_{X_{(i)}} (x_0) = \frac{n!}{(i-1)! (n-i)!} [F_X (x_0)]^{i-1} [1-F_X (x_0)]^{n-i} f_X (x_0).
\end{equation*}
Now, taking sum over $i$, we have
\begin{equation*}
\begin{split}
\sum_{i=1}^{n} f_{X_{(i)}} (x_0) &= \sum_{i=1}^{n} \frac{n!}{(i-1)! (n-i)!} [F_X (x_0)]^{i-1} [1-F_X (x_0)]^{n-i} f_X (x_0) \\
&= n \sum_{i=1}^{n} \frac{(n-1)!}{(i-1)! (n-i)!} [F_X (x_0)]^{i-1} [1-F_X (x_0)]^{n-i} f_X (x_0) \\
&= n [F_X (x_0) +1 -F_X (x_0)]^{n-1} f_X (x_0) = n f_X (x_0).
\end{split}
\end{equation*}
Hence, $\frac{1}{n} \sum_{i=1}^{n} f_{X_{(i)}} (x_0) = f_X (x_0)$ for any arbitrary $x_0$, which completes the proof. \hfill$\Box$

\vspace{0.25in}

\noindent {\bf Fact B:} {\it 
\begin{equation*}
AE(\hat{g}_{n,\alpha} (x), \hat{g}_{n, NW} (x)) = \frac{(1 - 2\alpha) t_{\alpha} (x)}{f_{X}(x)}\leq 1
\end{equation*} for any $x$ and $\alpha\in [0, \frac{1}{2})$. Here the notations are same as defined in Section 3.1.}

\noindent {\bf Proof of Fact B:} Using the form of the probability density function of $X_{(i)}$, we have 
\begin{equation}
\label{bin}
\begin{split}
t_{\alpha} (x) = \frac{1}{(n-2[n\alpha])} \sum\limits_{i=[n \alpha] +1}^{n - [n \alpha]} f_{X_{(i)}} (x) &= \frac{1}{(n-2[n\alpha])} \sum\limits_{i=[n \alpha] +1}^{n - [n \alpha]} n C(n-1,i-1) [F_{X}(x)]^{i-1} [1-F_{X}(x)]^{n-i} f_X (x) \\
&= \frac{n f_X (x)}{(n-2[n\alpha])} \sum\limits_{i=[n \alpha] +1}^{n - [n \alpha]} C(n-1,i-1) [F_{X}(x)]^{i-1} [1-F_{X}(x)]^{n-i},
\end{split}
\end{equation}
where $F_{X}$ is the distribution function of $X$ and $C(n-1,i-1) = \frac{(n - 1)!}{(i - 1)! (n - i)!}$. We also note that,
\begin{equation}
\label{bin2}
\sum\limits_{i=[n \alpha] +1}^{n - [n \alpha]} C(n-1,i-1) [F_{X}(x)]^{i-1} [1-F_{X}(x)]^{n-i} \leqslant \sum\limits_{i=1}^{n} C(n-1,i-1) [F_{X}(x)]^{i-1} [1-F_{X}(x)]^{n-i} =1,
\end{equation}
for all $n$ and any fixed $\alpha \in [0,\frac{1}{2})$. Hence, (\ref{bin}) and (\ref{bin2}) together give us 
\begin{equation*}
t_{\alpha} (x) = \lim_{n\to \infty} \frac{1}{(n-2[n\alpha])} \sum\limits_{i=[n \alpha] +1}^{n - [n \alpha]} f_{X_{(i)}} (x) \leqslant \lim_{n\to \infty} \frac{n f_X (x)}{(n-2[n\alpha])} = \frac{f_X (x)}{(1-2\alpha)}.
\end{equation*}
It completes the proof.\hfill$\Box$

\vspace{8mm}

In order to prove Theorem \ref{T1}, one needs the following lemmas. 
\vspace{5mm}

\begin{lemma}
\label{L1}
Let $\displaystyle \hat{f}_{n, \alpha} (x_0) = \frac{1}{(n - 2[n \alpha]) h_n} \sum\limits_{i = [n \alpha] +1}^{n -[n \alpha]} K\left( \frac{X_{(i)} -x_0}{h_n} \right)$, where $K$ is the kernel function. Then, under (A1)-(A5), for any arbitrary $x_{0}$, $\displaystyle \hat{f}_{n, \alpha} (x_0) \xrightarrow{\makebox[5mm]{p}} \lim_{n\to \infty} \frac{1}{(n-2[n\alpha])} \sum\limits_{i=[n \alpha] +1}^{n - [n \alpha]} f_{X_{(i)}} (x_0)$ as $n\to \infty$, where $f_{X_{(i)}}(\cdot)$ is the density function of $i$-th order statistic $X_{(i)}$.
\end{lemma}

\begin{proof}
To prove this lemma, it is enough to show that for any arbitrary $x_{0}$, $E[\hat{f}_{n, \alpha} (x_{0})]\rightarrow\displaystyle\lim_{n\rightarrow\infty}\frac{1}{(n - 2[n\alpha])}\sum\limits_{i =  [n\alpha]+1}^{n - [n\alpha]} f_{X_{i}}(x_{0})$ as $n\rightarrow\infty$ and variance$[\hat{f}_{n, \alpha} (x_{0})]\rightarrow 0$ as $n\rightarrow\infty$. 

\vspace{0.5in}

We now consider 
\begin{equation*}
\begin{split}
E\left(\hat{f}_{n, \alpha} (x_0)\right) &= \frac{1}{h_n (n - 2[n \alpha])} \sum\limits_{i=[n \alpha] +1}^{n - [n \alpha]} E\left[ K\left( \frac{X_{(i)}-x_0}{h_n} \right) \right] \\
&= \frac{1}{h_n (n - 2[n \alpha])} \sum\limits_{i=[n \alpha] +1}^{n - [n \alpha]} \int K\left( \frac{u-x_0}{h_n} \right) f_{X_{(i)}} (u) du
\end{split}
\end{equation*}
Using $\displaystyle \frac{u-x_0}{h_n} =z$ yields
\begin{equation*}
\begin{split}
E\left(\hat{f}_{n, \alpha} (x_0)\right) &= \frac{1}{h_n (n - 2[n \alpha])} \sum\limits_{i=[n \alpha] +1}^{n - [n \alpha]} h_n \int K(z) f_{X_{(i)}} (x_0 +h_n z) dz \\
&= \frac{1}{(n - 2[n \alpha])} \sum\limits_{i=[n \alpha] +1}^{n - [n \alpha]} \int K(z) f_{X_{(i)}} (x_0 +h_n z) dz \\
& \to \lim_{n\to \infty} \frac{1}{(n - 2[n \alpha])} \sum\limits_{i=[n \alpha] +1}^{n - [n \alpha]} f_{X_{(i)}} (x_0)
\end{split}
\end{equation*}
as $n\to \infty$, for each $x_0$. The last implication follows from the application of dominated convergence theorem using the facts that $f_{X_{(i)}}$ is a bounded function (follows from (A2)), $h_{n}\rightarrow 0$ as $n\rightarrow\infty$ and $\int\limits_{-\tau}^{\tau} k(z) dz = 1$.

Next, consider ($var$ denotes the variance and $cov$ denotes the co-variance)
\begin{equation}
\begin{split}
\label{L1e1}
var \left(\hat{f}_{n,\alpha} (x_0)\right) &= \frac{1}{h_n^2 (n - 2[n \alpha])^2} \sum\limits_{i=[n \alpha] +1}^{n - [n \alpha]} var\left[ K\left( \frac{X_{(i)}-x_0}{h_n} \right) \right] \\
& + \frac{1}{h_{n}^{2} (n - 2[n\alpha])^{2}}\sum\limits_{i\neq j = [n\alpha] +1}^{n - [n\alpha]} cov\left[K\left(\frac{X_{(i)}- x_{0}}{h_{n}}\right), K\left(\frac{X_{(j)} - x_{0}}{h_{n}}\right)\right]\\
&= \frac{1}{h_n^2 (n - 2[n \alpha])^2} \sum\limits_{i=[n \alpha] +1}^{n - [n \alpha]} E\left[ K\left( \frac{X_{(i)}-x_0}{h_n} \right)^2 \right] - E\left[\left\{K\left( \frac{X_{(i)}-x_0}{h_n} \right)\right\} \right]^2 \\
& + \frac{1}{h_{n}^{2}(n - 2[n\alpha])^{2}} \sum\limits_{i\neq j = [n\alpha] + 1}^{n - [n\alpha]} cov\left[K\left(\frac{X_{(i)} - x_{0}}{h_{n}}\right), K\left(\frac{X_{(j)} - x_{0}}{h_{n}}\right)\right]\\
&= \frac{1}{h_n^2 (n - 2[n \alpha])^2} \sum\limits_{i=[n \alpha] +1}^{n - [n \alpha]} \int \left[K\left( \frac{u-x_0}{h_n} \right)\right]^2 f_{X_{(i)}} (u) du \\
&- \frac{1}{h_n^2 (n - 2[n \alpha])^2} \sum\limits_{i=[n \alpha] +1}^{n - [n \alpha]} \left( \int K\left( \frac{u-x_0}{h_n} \right) f_{X_{(i)}} (u) du \right)^2\\
& + \frac{1}{h_{n}^{2}(n - 2[n\alpha])^{2}} \sum\limits_{i\neq j = [n\alpha] + 1}^{n - [n\alpha]} cov\left[K\left(\frac{X_{(i)} - x_{0}}{h_{n}}\right), K\left(\frac{X_{(j)} - x_{0}}{h_{n}}\right)\right].
\end{split}
\end{equation}

Note that since for any random variable Y, $var(Y) \leq E[Y^{2}]$ and $var(Y)\geq 0$ along with the fact that for any two random variables $X$ and $Y$, $|cov (X, Y)|\leq\sqrt{var(X) var(Y)}$, it is now enough to show that $\frac{1}{h_n^2 (n - 2[n \alpha])^2} \sum\limits_{i=[n \alpha] +1}^{n - [n \alpha]} \int \left[K\left( \frac{u-x_0}{h_n} \right)\right]^2 f_{X_{(i)}} (u) du\rightarrow 0$ as $n\rightarrow\infty$. Using the transformation $\displaystyle z=\frac{u-x_0}{h_n}$, we have
\begin{equation*}
\begin{split}
\frac{1}{h_n^2 (n - 2[n \alpha])^2} & \sum\limits_{i=[n \alpha] +1}^{n - [n \alpha]} \int \left[K\left( \frac{u-x_0}{h_n} \right)\right]^2 f_{X_{(i)}} (u) du \\ 
&= \frac{1}{h_n^2 (n - 2[n \alpha])^2} \sum\limits_{i=[n \alpha] +1}^{n - [n \alpha]} h_n \int \{K(z)\}^2 f_{X_{(i)}} (x_0 +h_n z) dz \\
&= \frac{1}{h_n (n - 2[n \alpha])} \frac{1}{(n - 2[n \alpha])} \sum\limits_{i=[n \alpha] +1}^{n - [n \alpha]} \int\{K(z)\}^2 f_{X_{(i)}} (x_0 +h_n z) dz.
\end{split}
\end{equation*}
Here since $nh_{n}\rightarrow\infty$ as $n\rightarrow\infty$, $\displaystyle \lim_{n\to \infty} h_n^{-1} (n - 2[n \alpha])^{-1} =0$ for a fixed $\alpha\in [0, \frac{1}{2})$, we have
\begin{equation*}
\frac{1}{(n - 2[n \alpha])} \sum\limits_{i=[n \alpha] +1}^{n - [n \alpha]} \int \{K(z)\}^2 f_{X_{(i)}} (x_0 +h_n z) dz \to \lim_{n\to \infty} \frac{1}{(n - 2[n \alpha])} \sum\limits_{i=[n \alpha] +1}^{n - [n \alpha]} f_{X_{i}} (x_0)
\end{equation*}
as $n\to \infty$ for each $x_0$ using the same argument provided for $E [\hat{f}_{n, \alpha} (x_{0})]$. Thus $$\displaystyle \frac{1}{h_n^2 (n - 2[n \alpha])^2} \sum\limits_{i=[n \alpha] +1}^{n - [n \alpha]} \int K\left( \frac{u-x_0}{h_n} \right)^2 f_{X_{i}} (u) du\to 0$$ as $n\to \infty$, and hence, 
$var \left(\hat{f}_{n, \alpha} (x_{0})\right)\rightarrow 0$ as $n\rightarrow\infty$. It completes the proof.

\end{proof}

\begin{lemma}
\label{L2}
Let
\begin{equation}
\hat{m}_{1, n} (x_0) = \frac{1}{ (n - 2[n \alpha]) h_n} \sum\limits_{i = [n \alpha] +1}^{ n - [n \alpha]} K \left( \frac{X_{(i)} -x_0}{h_n} \right)\left( g(X_{(i)}) - g(x_0) \right),
\end{equation}
where $K$ is the kernel function, $g$ is the regression function and $X_{(i)}$ is the $i$-th order statistic. Then under (A1)-(A5), we have
\begin{equation*}
E ( \hat{m}_{1, n} (x_0) ) = \frac{h_n^2 g^{\prime \prime} (x_0)\, k_2}{2 \left( n - 2[n \alpha]\right)} \sum\limits_{i = [n \alpha] +1}^{n - [n \alpha]} f_{X_{(i)}} (x_0) + \frac{h_n^2 g^{\prime} (x_0)\, k_2}{\left( n - 2[n \alpha]\right)} \sum\limits_{i = [n \alpha] +1}^{n - [n \alpha]} f^{\prime}_{X_{(i)}} (x_0) +o(h_n^2)
\end{equation*}
for each $x_0$, where $g^{\prime}$ and $g^{\prime \prime}$ are the notations for first and second derivatives of $g$, respectively, $f_{X_{(i)}} (\cdot)$ is the probability density function of $i$-th order statistic $X_{(i)}$ with its derivative $f_{X_{(i)}}^{'}$  and $\displaystyle k_2 = \int v^2K(v) dv$.
\end{lemma}

\begin{proof}
Note that 
\begin{equation*}
\begin{split}
E ( \hat{m}_{1, n} (x_0) ) & = E \left( \frac{1}{(n - 2[n \alpha]) h_n } \sum\limits_{i = [n \alpha] +1}^{n - [n \alpha]} K \left(\frac{X_{(i)} - x_0}{h_n} \right) \left( g(X_{(i)}) - g(x_0) \right) \right) \\
& = \frac{1}{( n - 2[n \alpha]) h_n} \sum\limits_{i = [n \alpha] +1}^{n - [n \alpha]} \int K \left( \frac{u-x_0}{h_n} \right) \left( g(u)-g(x_0)\right) f_{X_{(i)}}(u) du .
\end{split}
\end{equation*}
Using transformation $\displaystyle v = \frac{u-x_0}{h_n}$ and (A3) $\left(i.e., \displaystyle \int v K(v) dv =0\right)$, we have
\begin{equation*}
\begin{split}
& \int K \left( \frac{u-x_0}{h_n} \right) h_n \left( g(u)-g(x_0) \right) f_{X_{(i)}} (u) du \\
& = h_n \int K(v) \left( g(x_0 + v h_n) - g(x_0) \right) f_{X_{(i)}}(x_0 + v h_n) dv \\
& = h_n \int K(v) \left( h_n v g^{\prime}(x_0) + \frac{h_n^2 v^2}{2} g^{\prime \prime}(x_0) + o(h_n^2) \right) \left( f_{X_{(i)}} (x_0) + h_n v f^{\prime}_{X_{(i)}} (x_0) + o(h_n) \right) dv \\
& = h_n \int K(v) \left( h_n v g^{\prime}(x_0) f_{X_{(i)}}(x_0) +  \frac{h_n^2 v^2}{2} g^{\prime \prime} f_{X_{(i)}}(x_0) + h_n^2 v^2 g^{\prime} f^{\prime}_{X_{(i)}}(x_0) + o(h_n^2) \right) dv \\
& = \frac{1}{2} h_n^3 g^{\prime \prime}(x_0) f_{X_{(i)}} (x_0) \int v^2K(v) dv + h_n^2 g^{\prime}(x_0) f^{\prime}_{X_{(i)}} (x_0) \int v^2K(v) dv + o(h_n^3) \\
& = h_n^3 k_2 \left(\frac{1}{2} g^{\prime \prime}(x_0) f_{X_{(i)}} (x_0) + g^{\prime}(x_0) f^{\prime}_{X_{(i)}} (x_0) \right) + o(h_n^3)
\end{split}
\end{equation*}
for any arbitrary $x_0$, where $\displaystyle k_2 = \int v^2K(v) dv$. Hence, we have 
\begin{equation*}
E (\hat{m}_{1, n} (x_0)) = \frac{h_n^2 g^{\prime \prime} (x_0)\, k_2}{2 \left( n - 2[n \alpha]\right)} \sum\limits_{i = [n \alpha] +1}^{n - [n \alpha]} f_{X_{(i)}} (x_0) + \frac{h_n^2 g^{\prime} (x_0)\, k_2}{\left( n - 2[n \alpha]\right)} \sum\limits_{i = [n \alpha] +1}^{n - [n \alpha]} f^{\prime}_{X_{(i)}} (x_0) +o(h_n^2).
\end{equation*}
It completes the proof. 
\end{proof}

\begin{lemma}
\label{L3}
Under (A1)-(A5),
\begin{equation*}
var \left( \sqrt{n h_n } \hat{m}_{1, n} (x_0) \right) \longrightarrow 0 \text{ as } n \to \infty, 
\end{equation*} where $\hat{m}_{1, n} (x_{0})$ is same as defined in the statement of Lemma \ref{L2}. 
\end{lemma}

\begin{proof}
To prove the assertion of this lemma, we find an upper bound of $\sqrt{n h_n} \hat{m}_{1, n} (x_0)$ as follows:
\begin{equation*}
\begin{split}
\left| \sqrt{n h_n} \hat{m}_{1, n} (x_0) \right| & = \left| \frac{\sqrt{n h_n} }{ (n - 2[n \alpha]) h_n} \sum\limits_{i = [n \alpha] +1}^{ n - [n \alpha]} K \left( \frac{X_{(i)} - x_0}{h_n} \right) \left( g(X_{(i)}) - g(x_0) \right) \right| \\
& \leqslant \frac{\sqrt{n}}{(n - 2[n \alpha]) \sqrt{h_n}} \sum\limits_{i = [n \alpha] +1}^{ n - [n \alpha]} \left| K \left( \frac{X_{(i)} -x_0}{h_n} \right) \right| \left| \left( g (X_{(i)}) - g(x_0) \right) \right| \\
& = \frac{\sqrt{n}}{ (n - 2[n \alpha]) \sqrt{h_n}} \sum\limits_{i = [n \alpha] +1}^{ n - [n \alpha]} M h_n \left| X_{(i)} - x_0 \right| \left|  g^{\prime} (\xi) \right| ,
\end{split}
\end{equation*}
using first order Tailor series expansion of $g(\cdot)$, where $\xi$ lies between $X_{(i)}$ and $x_0$, and $M$ is an arbitrary constant such that $h_{n}^{-1} k(./h_{n})\leq M$ (see assumption (A3)). Thus we have,
\begin{equation*}
\left| \sqrt{n h_n} \hat{m}_{1, n} (x_0) \right| \leqslant \frac{\sqrt{n}}{(n - 2[n \alpha]) \sqrt{h_n}} \sum\limits_{i = [n \alpha] +1}^{ n - [n \alpha]} M h_n M_1 \frac{\log n}{n} \leqslant \sqrt{h_{n}} M M_1 \frac{\log n}{\sqrt{n}} ,
\end{equation*}
using spacing property of order statistics (see \cite{devroye1981laws}), where $M_1$ is an upper bound for $g^{\prime}$ as $g^{'}$ is bounded, which follows from (A1). It leads to $\displaystyle \sqrt{n h_n}  \hat{m}_{1, n} (x_0) \xrightarrow{\makebox[.5cm]{a.s}} 0$ as $n \to \infty $, which implies that $\displaystyle \left( \sqrt{n h_n}  \hat{m}_{1, n} (x_0)\right)^2 \xrightarrow{\makebox[.5cm]{a.s}} 0$ as $n \to \infty $. Using similar arguments, one can establish that $\displaystyle \left( \sqrt{n h_{n}} \hat{m}_{1, n} (x_0)\right)^2 $ is uniformly bounded in probability. Hence by Dominated Convergence Theorem, $\displaystyle E \left[ \left(\sqrt{n h_{n}} \hat{m}_{1, n} (x_0)\right)^2 \right] \to 0$ as $n \to \infty $, which implies that $var \left( \sqrt{n h_n } \hat{m}_{1, n} (x_0) \right) \to 0 \text{ as } n \to \infty$.

\end{proof}

\begin{lemma}
\label{L4}
Let
\begin{equation}
\hat{m}_{2, n} (x_0) = \frac{1}{ (n - 2[n \alpha]) h_n} \sum\limits_{i = [n \alpha] +1}^{ n - [n \alpha]} K \left(\frac{X_{(i)} -x_0}{h_n} \right) e_{[i]},
\end{equation}
where $K$ is the kernel function, $X_{(i)}$ is the $i$-th order statistic and $e_{[i]}$ denotes the error corresponding to $X_{(i)}$. Then under (A1)-(A5),
\begin{equation*}
E \left( \sqrt{n h_n} \hat{m}_{2, n} (x_0) \right) = 0 \text{  for any arbitrary $x_0$ and for all $n$}.
\end{equation*}
\end{lemma}

\begin{proof}
Assumption (A5) implies that $E(e_{[i]} | X_{(i)}) = 0$, which further infers that 
\begin{equation*}
E\left( K\left( \frac{X_{(i)} - x_0}{h_n} \right) e_{[i]} \right) = E\left( K\left( \frac{X_{(i)} - x_0}{h_n} \right) E\left( e_{[i]} \big| X_{(i)} \right) \right) = 0,
\end{equation*}
for each $x_0$ and $i=1,\ldots ,n$. Using this fact, we now have 
\begin{equation}
E \left( \sqrt{n h_n} \hat{m}_{2, n} (x_0) \right) = \frac{\sqrt{n h_n}}{ (n - 2[n \alpha]) h_n} \sum\limits_{i = [n \alpha] +1}^{ n - [n \alpha]} E \left( K \left(\frac{X_{(i)} -x_0}{h_n} \right) e_{[i]}\right) = 0.
\end{equation}
It completes the proof.
\end{proof}

\begin{lemma}
\label{L5}
Under (A1)-(A5),
\begin{equation*}
var \left( \sqrt{n h_n}\, \, \hat{m}_{2, n} (x_0) \right) \longrightarrow \frac{\sigma^2 (x_0) t_{\alpha} (x_0) \int \{K(z)\}^2 dz}{(1-2\alpha)} \quad \text{as } n\to \infty
\end{equation*}
for any arbitrary $x_0$, where $\hat{m}_{2, n} (x_{0})$ is same as defined in the statement of Lemma \ref{L4}. Here  
$\displaystyle \sigma^2 (x_0) = E\left(e^{2} | X = x_0 \right)$ and $\displaystyle t_{\alpha} (x_0) = \lim_{n\to \infty} \frac{1}{(n - 2[n \alpha])} \sum\limits_{i = [n \alpha] +1}^{ n - [n \alpha]} f_{X_{(i)}} (x_0)$.
\end{lemma}

\begin{proof}
First note that $\displaystyle var \left( \sqrt{n h_n}\, \, \hat{m}_{2,n} (x_0) \right) = E \left[ \left( \sqrt{n h_n}\, \, \hat{m}_{2,n} (x_0) \right)^2 \right]$, which follows from the assertion of Lemma \ref{L4}. We now have

\vspace{0.25in}

\begin{equation}
\begin{split}
& var \left( \sqrt{n h_n} \hat{m}_{2,n} (x_0) \right) = \frac{n h_n}{ (n - 2[n \alpha])^2 h_n^2} E\left[ \left( \sum\limits_{i = [n \alpha] +1}^{ n - [n \alpha]} K \left(\frac{X_{(i)} -x_0}{h_n} \right) e_{[i]} \right)^2 \right] \\
&= \frac{n h_n}{ (n - 2[n \alpha])^2 h_n^2} E \Bigg[ \sum\limits_{i = [n \alpha] +1}^{ n - [n \alpha]} K \left(\frac{X_{(i)} -x_0}{h_n} \right)^2 e_{[i]}^2
+ \sum\limits_{\substack{i,j = [n \alpha] +1 \\ i\neq j}}^{ n - [n \alpha]} K \left(\frac{X_{(i)} -x_0}{h_n} \right) e_{[i]} K \left(\frac{X_{(j)} -x_0}{h_n} \right) e_{[j]} \Bigg].
\end{split}
\end{equation}

\vspace{0.25in}

We now consider
\begin{equation*}
\begin{split}
E \left( K \left(\frac{X_{(i)} -x_0}{h_n} \right)^2 (e_{[i]})^2 \right) &= E\left( K\left( \frac{X_{(i)} - x_0}{h_n} \right)^2 E\left( e_{[i]}^2 \, \big|\, X_{(i)} \right) \right) \\
&= E\left( K\left( \frac{X_{(i)} - x_0}{h_n} \right)^2 \sigma^2 (X_{(i)}) \right) \\
&= \int K\left( \frac{u - x_0}{h_n} \right)^2 \sigma^2 (u) f_{X_{(i)}} (u) du .
\end{split}
\end{equation*}
Substituting $\displaystyle \frac{u - x_0}{h_n} = z$, and applying Taylor expansion in the neighbourhood of $x_0$, we have
\begin{equation}
\begin{split}
\label{L5e2}
\int & K\left( \frac{u - x_0}{h_n} \right)^2 \sigma^2 (u) f_{X_{(i)}} (u) du = h_n \int \{K\left( z \right)\}^2 \sigma^2 (x_0 + h_n z) f_{X_{(i)}} (x_0 + h_n z) dz \\
&= h_n \int \{K\left( z \right)\}^2 \left(\sigma^2(x_0)+ h_n z(\sigma^2)^{\prime} (x_0) + o(h_n)\right) \left( f_{X_{(i)}} (x_0) + h_n z f^{\prime}_{X_{(i)}} (x_0) + o(h_n) \right) dz \\
&= h_n \sigma^2 (x_0) f_{X_{(i)}} (x_0) \int \{K(z)\}^2 dz + o(h_n).
\end{split}
\end{equation}
Next for the product term, without loss of generality, one can consider the summand based on the sample $(X_1,\ldots ,X_{n-2[n\alpha]})$ and the corresponding errors $(e_1,\ldots ,e_{n-2[n\alpha]})$. Then the expectation $E \left[ K \left(\frac{X_i -x_0}{h_n}\right) e_i K \left(\frac{X_j -x_0}{h_n}\right) e_j \right] =0$, for each pair $\{(i,j)| i,j=1,\ldots ,n-2[n\alpha] ; i\neq j \}$, using the fact that $E(e_i|X_i) =0$ for all $i =1,2\ldots ,n-2[n\alpha]$. Hence, we have 
\begin{equation}
\label{L5e3}
E\left[\sum\limits_{\substack{i,j = [n \alpha] +1 \\ i\neq j}}^{ n - [n \alpha]} K \left(\frac{X_{(i)} -x_0}{h_n} \right) e_{[i]} K \left(\frac{X_{(j)} -x_0}{h_n} \right) e_{[j]}\right] =0
\end{equation}
Finally, combining (\ref{L5e2}) and (\ref{L5e3}), we have
\begin{equation*}
\begin{split}
var \left( \sqrt{n h_n}\, \, \hat{m}_{2, n} (x_0) \right) &= \frac{n h_n}{ (n - 2[n \alpha])^2 h_n^2} \sum\limits_{i = [n \alpha] +1}^{ n - [n \alpha]} h_n \sigma^2 (x_0) f_{X_{(i)}} (x_0) \int\{K(z)\}^2 dz + o(h_n) \\
&= \int\{K(z)\}^2 dz \frac{n \sigma^2 (x_0)}{(n - 2[n \alpha])^2} \sum\limits_{i = [n \alpha] +1}^{ n - [n \alpha]} f_{X_{(i)}} (x_0) + o(1).
\end{split}
\end{equation*}
Note that since $\displaystyle \frac{n}{(n - 2[n \alpha])} \to \frac{1}{(1-2\alpha)}$ as $n\to \infty$, we have
\begin{equation}
var \left( \sqrt{n h_n}\, \, \hat{m}_{2, n} (x_0) \right) \longrightarrow \frac{\sigma^2 (x_0) t_{\alpha} (x_0) \int \{K(z)\}^2 dz}{(1-2\alpha)} \quad \text{as } n\to \infty
\end{equation}
for any arbitrary $x_0$, where $\displaystyle t_{\alpha} (x_0) = \lim_{n\to \infty} \frac{1}{(n - 2[n \alpha])} \sum\limits_{i = [n \alpha] +1}^{ n - [n \alpha]} f_{X_{(i)}} (x_0)$. It completes the proof.
\end{proof}

\begin{lemma}
\label{L6}
Under (A2)-(A6), $\displaystyle \sqrt{nh_n} \left[\hat{m}_{2,n}(x_0) -E(\hat{m}_{2,n}(x_0)) \right] \xrightarrow{\makebox[5mm]{d}} \mathcal{N} (0,V)$, where $\hat{m}_{2, n} (x_{0})$ is same as defined in the statement of Lemma \ref{L4} and $\displaystyle V = \frac{\sigma^2 (x_0) }{(1-2\alpha) t_{\alpha} (x_0)} \int \{K(z)\}^2 dz$. Especially, since $E(\hat{m}_{2,n}(x_0)) = 0$ (see Lemma \ref{L4}), we have $\sqrt{nh_n} \hat{m}_{2,n}(x_0) \xrightarrow{\makebox[5mm]{d}} \mathcal{N} (0,V)$.
\end{lemma}

\begin{proof}
Note that 
\begin{equation}
\begin{split}
\label{L6e1}
\hat{m}_{2, n} (x_0) &= \frac{1}{ (n - 2[n \alpha]) h_n} \sum\limits_{i = [n \alpha] +1}^{ n - [n \alpha]} K \left(\frac{X_{(i)} -x_0}{h_n} \right) e_{[i]} \\
&= \frac{n}{(n - 2[n \alpha])} \Bigg[ \frac{1}{nh_n} \sum_{i=1}^{n} K \left(\frac{X_{(i)} -x_0}{h_n} \right) e_{[i]} \\
&- \frac{1}{nh_n} \sum_{i=1}^{[n\alpha]} K \left(\frac{X_{(i)} -x_0}{h_n} \right) e_{[i]} \\
&- \frac{1}{nh_n} \sum_{i=n-[n\alpha]+1}^{n} K \left(\frac{X_{(i)} -x_0}{h_n} \right) e_{[i]} \Bigg]\\
&:= \frac{n}{(n - 2[n \alpha])} [T_{n, 1} - T_{n, 2} - T_{n, 3}]
\end{split}
\end{equation}
We now investigate the distributional properties of $T_{n, 1}$, $T_{n, 2}$ and $T_{n, 3}$ in (\ref{L6e1}) separately.  
We now consider $T_{n, 1}$, which is the following. 
\begin{equation*}
T_{n, 1} = \frac{1}{nh_n} \sum_{i=1}^{n} K \left(\frac{X_{(i)} -x_0}{h_n} \right) e_{[i]} = \frac{1}{nh_n} \sum_{i=1}^{n} K \left(\frac{X_i -x_0}{h_n} \right) e_i .
\end{equation*}
The last equality follows from the definitions of $X_{(i)}$ and $e_{[i]}$.

Note that for $i = 1,\ldots, n$, $K\left(\frac{X_i -x_0}{h_n} \right) e_i$ is a sequence of independent random variables, with zero mean and variances given by
\begin{equation*}
\begin{split}
var \left( K\left(\frac{X_i -x_0}{h_n}\right)e_i \right) &= E\left[ \left\{K\left(\frac{X_i -x_0}{h_n}\right)\right\}^2 e_i^2 \right] - E\left[ K\left(\frac{X_i -x_0}{h_n}\right) e_i \right]^2 \\
&= E\left[ K\left(\frac{X_i -x_0}{h_n}\right)^2 e_i^2 \right] = E\left[ K\left(\frac{X_i -x_0}{h_n}\right)^2 E[e_i^2 | X_i] \right] \\
&= E\left[ K\left(\frac{X_i -x_0}{h_n}\right)^2 \sigma^2 (X_i) \right] = \int K\left(\frac{u -x_0}{h_n}\right)^2 \sigma^2 (u) f_X (u) du.
\end{split}
\end{equation*}
Now, changing the variable to $z = \frac{u-x_0}{h_n}$, we have
\begin{equation*}
\begin{split}
var \left( K\left(\frac{X_i -x_0}{h_n}\right)e_i \right) &= h_n \int\{K(z)\}^2 \sigma^2 (x_0 +h_n z) f_X (x_0 +h_n z) dz = h_n c(x_0)+o(h_n),
\end{split}
\end{equation*}
where $c(x_0) = \sigma^2(x_0) f_X (x_0) \int \{K(z)\}^2 dz <\infty$. The sum of variances is then given by 
\begin{equation*}
s_n = \sum_{i=1}^{n} var \left( K\left(\frac{X_i -x_0}{h_n}\right)e_i \right) = \sum_{i=1}^{n} h_n c(x_0) = nh_n c(x_0) +o(h_n).
\end{equation*}
Now, using condition (A6) in the Lyapunov condition for the sequence of random variables $K\left(\frac{X_i -x_0}{h_n} \right) e_i$ and for some $\delta >0$ we have,
\begin{equation*}
\frac{1}{s_n^{2+\delta}} \sum_{i=1}^{n} E\left[\left|K\left(\frac{X_i -x_0}{h_n}\right)e_i \right|^{2+\delta} \right] = \frac{1}{s_n^{2+\delta}} \sum_{i=1}^{n} c_1 E[|e_i^{2+\delta}|X_i] = \frac{n}{s_n^{2+\delta}} c_2,
\end{equation*}
using boundedness of $K$ where $c_1$ and $c_2$ are constants. Note that, $\frac{n}{(nh_n)^{2+\delta}} \to 0$ as $n\to \infty$, using $nh_n^2 \to 0$ as $n\to \infty$ (see (A4)). Hence, this sequence of random variables are satisfying the condition of Lyapunov CLT (see, e.g., \citeA{billingsley1995probability}, p.362) in view of (A3), (A4) and (A6). Hence, by Lyapunove CLT, we have  
\begin{equation}
\label{ft}
\sqrt{nh_n}\left(\frac{1}{nh_n} \sum_{i=1}^{n} K \left(\frac{X_{i} -x_0}{h_n} \right) e_{i}\right) = \sqrt{nh_{n}}\left(\frac{1}{nh_n} \sum_{i=1}^{n} K \left(\frac{X_{(i)} -x_0}{h_n} \right) e_{[i]}\right)
\end{equation} converges weakly to a Gaussian distribution. 

\noindent Now, $T_{n, 2}$ in (\ref{L6e1}), we have 
\begin{equation}
\frac{1}{nh_n} \sum_{i=1}^{[n\alpha]} K \left(\frac{X_{(i)} -x_0}{h_n} \right) e_{[i]} = \frac{[n\alpha]}{n} \left( \frac{1}{[n\alpha]h_n} \sum_{i=1}^{[n\alpha]} K \left(\frac{X_{(i)} -x_0}{h_n} \right) e_{[i]} \right).
\end{equation}
Since $\frac{[n\alpha]}{n} \to \alpha$ as $n\to \infty$ and $[n\alpha] \to \infty$ as $n\to \infty$, using a similar argument as for the first term we conclude that for $m_n =[n\alpha]$
\begin{equation}
\label{st}
T_{n, 2} = \sqrt{nh_n}\left(\frac{1}{m_{n}h_n} \sum_{i=1}^{m_n} K \left(\frac{X_{(i)} -x_0}{h_n} \right) e_{[i]}\right)
\end{equation} converges weakly to another Gaussian distribution. 

\noindent Finally, for $T_{n, 3}$ in (\ref{L6e1}), we have
\begin{equation}
\label{tt}
T_{n, 3} = \frac{1}{nh_n} \sum_{i=n-[n\alpha]+1}^{n} K \left(\frac{X_{(i)} -x_0}{h_n} \right) e_{[i]} = \frac{1}{nh_n} \sum_{i=1}^{n} K \left(\frac{X_{(i)} -x_0}{h_n} \right) e_{[i]} - \frac{1}{nh_n} \sum_{i=1}^{n-[n\alpha]} K \left(\frac{X_{(i)} -x_0}{h_n} \right) e_{[i]}.
\end{equation}
Using a similar arguments as the first and the second terms, it also converges weakly to a another Gaussian distribution. Therefore, combining the asymptotic distributions of the expressions of (\ref{ft}), (\ref{st}) and (\ref{tt}) leads to the asymptotic normality of $\hat{m}_{2, n} (x_{0})$, which completes the proof.

\end{proof}

\vspace{8mm}
\noindent {\bf Proof of Theorem \ref{T1}:}

\begin{proof}

The regression model defined in (\ref{model}) can be rewritten as
\begin{equation*}
\begin{split}
Y_i = g (X_i) + e_i = g(x_0) + (g(X_i) - g(x_0)) + e_i, \text{ for } i = 1,2,\ldots ,n \\
\text{OR} \quad Y_{[i]} = g(x_0) + (g(X_{(i)}) - g(x_0)) + e_{[i]}, \text{ for } i = 1,2,\ldots ,n ,
\end{split}
\end{equation*} 
with the ordered version of random variables. Here $X_{(i)}$ denotes the $i$-th order statistic of $\{X_1,\ldots ,X_n\}$, and $Y_{[i]}$ and $e_{[i]}$ are the corresponding response and error random variables (as defined in Section 2). Note that 
\begin{equation}
\hat{g}_{n,\alpha} (x_0) = g(x_0) + \frac{\hat{m}_{1, n} (x_{0})}{\hat{f}_{n, \alpha} (x_{0})} + \frac{\hat{m}_{2, n} (x_{0})}{\hat{f}_{n, \alpha} (x_{0})}, 
\end{equation} where $\hat{f}_{n, \alpha} (x_{0})$, $\hat{m}_{1, n} (x_{0})$ and $\hat{m}_{2, n} (x_{0})$ are same as defined in Lemmas \ref{L1}, \ref{L2} and \ref{L4}, respectively. 

Note that the assertions in Lemmas \ref{L1}, \ref{L4}, \ref{L5} and \ref{L6} imply that 
\begin{equation}
\sqrt{n h_n} \frac{\hat{m}_{2,n} (x_0)}{\hat{f}_{n, \alpha} (x_{0})}  \overset{d}{\longrightarrow} \mathcal{N} \left( 0, \frac{\sigma^2 (x_0)}{(1-2\alpha)t_{\alpha} (x_{0})} \int \{K(z)\}^2 dz \right).
\end{equation}

The above two facts along with an application of Slutsky's theorem (see, e.g., \citeA{serfling2009approximation}), one can conclude that $$\sqrt{nh_n} \left( \hat{g}_{n,\alpha} (x_{0}) - g(x_{0}) - h_n^2 k_2 \left( \frac{g^{\prime \prime} (x_{0})}{2(n-2[n\alpha])} \sum\limits_{i=[n \alpha] +1}^{n - [n \alpha]} f_{X_{(i)}} (x_{0}) + \frac{g^{\prime}(x_{0})}{(n-2[n\alpha])} \sum\limits_{i=[n \alpha] +1}^{n - [n \alpha]} f^{\prime}_{X_{(i)}} (x_{0}) \right) \right)$$ converges weakly to a Gaussian distribution with mean $= 0$ and variance $= V$, where $V$ is same as defined in the statement of Theorem \ref{T1}.    

\end{proof}


\subsection{APPENDIX B}

\noindent {\bf Asymptotic Efficiency Table:} Corresponding to Section \ref{AE}.

\begin{table}[H]
 \begin{center}
 
 \caption{Table showing Asymptotic Efficiency of $\hat{g}_{n,\alpha} (0.5)$ relative to $\hat{g}_{n, NW} (0.5)$ for different values of $\alpha$.}
 \label{table_ae}
 \begin{tabular}{c c c c c c}
 
 \hline
 \multicolumn{6}{c}{$X$ follows $Unif(0,1)$} \\ 
 \hline 
$\alpha$ = 0.05 & $\alpha$ = 0.10 & $\alpha$ = 0.20 & $\alpha$ = 0.30 & $\alpha$ = 0.40 & $\alpha$ = 0.45 \\

 0.9782609 & 0.9836066 & 0.9999998 & 0.9673888 & 0.9153650 & 0.5826260 \\

 \multicolumn{6}{c}{$X$ follows $Beta(2,2)$} \\ 

$\alpha$ = 0.05 & $\alpha$ = 0.10 & $\alpha$ = 0.20 & $\alpha$ = 0.30 & $\alpha$ = 0.40 & $\alpha$ = 0.45 \\

0.9782609 & 1.0000000 & 0.9999752 & 0.9950485 & 0.8399297 & 0.5005812 \\
\hline

\end{tabular}
\end{center}
\end{table}

\noindent {\bf Finite Sample Study Table:} Corresponding to Section \ref{FSS}.

\begin{table}[H]
 \begin{center}
 
 \caption{Table showing the finite sample Efficiency of $\hat{g}_{n,\alpha} (0.5)$ relative to $\hat{g}_{n,NW} (0.5)$ for Examples 1, 2, 3 and 4 in Section \ref{FSS}. Here $n = 50$ and $500$.}
 \label{table_fss}
 \begin{tabular}{c c c c c c}
 
 \hline
 \multicolumn{6}{c}{Example 1 with $n=50$} \\ 
 \hline 
 
$\alpha$ = 0.05 & $\alpha$ = 0.10 & $\alpha$ = 0.20 & $\alpha$ = 0.30 & $\alpha$ = 0.40 & $\alpha$ = 0.45 \\

 1.0274176 & 0.9659629 & 1.0141424 & 0.6229415 & 0.2438898 & 0.1900028 \\
 
 \multicolumn{6}{c}{Example 1 with $n=500$} \\ 

$\alpha$ = 0.05 & $\alpha$ = 0.10 & $\alpha$ = 0.20 & $\alpha$ = 0.30 & $\alpha$ = 0.40 & $\alpha$ = 0.45 \\

 1.0822063 & 1.0020056 & 1.0444929 & 1.0289262 & 0.4826861 & 0.2658498 \\

\multicolumn{6}{c}{Example 2 with $n=50$} \\ 

$\alpha$ = 0.05 & $\alpha$ = 0.10 & $\alpha$ = 0.20 & $\alpha$ = 0.30 & $\alpha$ = 0.40 & $\alpha$ = 0.45 \\

1.0484987 & 0.9574283 & 0.9581104 & 0.7012951 & 0.3207287 & 0.2309493 \\

 \multicolumn{6}{c}{Example 2 with $n=500$} \\ 

$\alpha$ = 0.05 & $\alpha$ = 0.10 & $\alpha$ = 0.20 & $\alpha$ = 0.30 & $\alpha$ = 0.40 & $\alpha$ = 0.45 \\

0.9379374 & 0.9479973 & 1.0495048 & 1.0015816 & 0.5939786 & 0.3355768 \\

 \multicolumn{6}{c}{Example 3 with $n=50$} \\ 

$\alpha$ = 0.05 & $\alpha$ = 0.10 & $\alpha$ = 0.20 & $\alpha$ = 0.30 & $\alpha$ = 0.40 & $\alpha$ = 0.45 \\

1.0300403 & 1.0345411 & 0.8658538 & 0.6145255 & 0.2826788 & 0.2157184 \\

 \multicolumn{6}{c}{Example 3 with $n=500$} \\ 

$\alpha$ = 0.05 & $\alpha$ = 0.10 & $\alpha$ = 0.20 & $\alpha$ = 0.30 & $\alpha$ = 0.40 & $\alpha$ = 0.45 \\

1.0728524 & 0.9986963 & 1.1596102 & 1.0221897 & 0.4859994 & 0.2524595 \\

 \multicolumn{6}{c}{Example 4 with $n=50$} \\

$\alpha$ = 0.05 & $\alpha$ = 0.10 & $\alpha$ = 0.20 & $\alpha$ = 0.30 & $\alpha$ = 0.40 & $\alpha$ = 0.45 \\

0.9551044 & 0.9540824 & 0.9748282 & 0.7162221 & 0.3683285 & 0.2436081 \\

 \multicolumn{6}{c}{Example 4 with $n=500$} \\

$\alpha$ = 0.05 & $\alpha$ = 0.10 & $\alpha$ = 0.20 & $\alpha$ = 0.30 & $\alpha$ = 0.40 & $\alpha$ = 0.45 \\

1.0058514 & 1.0806839 & 1.0892772 & 1.0036713 & 0.5958239 & 0.3385112 \\
\hline

\end{tabular}
\end{center}
\end{table}

\noindent {\bf Real Data Analysis Table:} Corresponding to Section \ref{rda}.

\begin{table}[H]
 \begin{center}
 
 \caption{Table showing Bootstrap Efficiency of $\hat{g}_{n,\alpha} (0.5)$ relative to $\hat{g}_{n, NW} (0.5)$ for three benchmark Real Data studied in Section \ref{rda}.}
 \label{table_rda}
 \begin{tabular}{c c c c c c}
 
 \hline
 \multicolumn{6}{c}{Data: Combined Cycle Power Plant Data Set} \\ 
 \hline 
$\alpha$ = 0.05 & $\alpha$ = 0.10 & $\alpha$ = 0.20 & $\alpha$ = 0.30 & $\alpha$ = 0.40 & $\alpha$ = 0.45 \\

 0.6501292 & 0.6595586 & 0.7475265 & 1.0328150 & 0.2786650 & 0.1804794 \\

 \multicolumn{6}{c}{Data: Parkinson's Telemonitoring Data Set} \\ 

$\alpha$ = 0.05 & $\alpha$ = 0.10 & $\alpha$ = 0.20 & $\alpha$ = 0.30 & $\alpha$ = 0.40 & $\alpha$ = 0.45 \\

0.9309987 & 1.0066488 & 0.9618025 & 0.6231914 & 0.2356732 & 0.1749779 \\

 \multicolumn{6}{c}{Data: Air Quality Data Set} \\ 

$\alpha$ = 0.05 & $\alpha$ = 0.10 & $\alpha$ = 0.20 & $\alpha$ = 0.30 & $\alpha$ = 0.40 & $\alpha$ = 0.45 \\

1.0120231 & 0.9598871 & 0.6772431 & 0.4064943 & 0.0923924 & 0.0 \\
\hline

\end{tabular}
\end{center}
\end{table}

\bibliographystyle{apacite}
\bibliography{trimmed_nw_est}

\begin{thebibliography}{}

\bibitem [\protect \citeauthoryear {%
Bickel%
}{%
Bickel%
}{%
{\protect \APACyear {1965}}%
}]{%
bickel1965some}
\APACinsertmetastar {%
bickel1965some}%
\begin{APACrefauthors}%
Bickel, P\BPBI J.%
\end{APACrefauthors}%
\unskip\
\newblock
\APACrefYearMonthDay{1965}{}{}.
\newblock
{\BBOQ}\APACrefatitle {On some robust estimates of location} {On some robust
  estimates of location}.{\BBCQ}
\newblock
\APACjournalVolNumPages{The Annals of Mathematical
  Statistics}{36}{3}{847--858}.
\PrintBackRefs{\CurrentBib}

\bibitem [\protect \citeauthoryear {%
Billingsley%
}{%
Billingsley%
}{%
{\protect \APACyear {1995}}%
}]{%
billingsley1995probability}
\APACinsertmetastar {%
billingsley1995probability}%
\begin{APACrefauthors}%
Billingsley, P.%
\end{APACrefauthors}%
\unskip\
\newblock
\APACrefYear{1995}.
\newblock
\APACrefbtitle {Probability and measure} {Probability and measure}.
\newblock
\APACaddressPublisher{}{John Wiley \& Sons}.
\PrintBackRefs{\CurrentBib}

\bibitem [\protect \citeauthoryear {%
{\v{C}}{\'\i}{\v{z}}ek%
}{%
{\v{C}}{\'\i}{\v{z}}ek%
}{%
{\protect \APACyear {2016}}%
}]{%
vcivzek2016generalized}
\APACinsertmetastar {%
vcivzek2016generalized}%
\begin{APACrefauthors}%
{\v{C}}{\'\i}{\v{z}}ek, P.%
\end{APACrefauthors}%
\unskip\
\newblock
\APACrefYearMonthDay{2016}{}{}.
\newblock
{\BBOQ}\APACrefatitle {Generalized method of trimmed moments} {Generalized
  method of trimmed moments}.{\BBCQ}
\newblock
\APACjournalVolNumPages{Journal of Statistical Planning and
  Inference}{171}{}{63--78}.
\PrintBackRefs{\CurrentBib}

\bibitem [\protect \citeauthoryear {%
Clark%
}{%
Clark%
}{%
{\protect \APACyear {1977}}%
}]{%
clark1977non}
\APACinsertmetastar {%
clark1977non}%
\begin{APACrefauthors}%
Clark, R.%
\end{APACrefauthors}%
\unskip\
\newblock
\APACrefYearMonthDay{1977}{}{}.
\newblock
{\BBOQ}\APACrefatitle {Non-parametric estimation of a smooth regression
  function} {Non-parametric estimation of a smooth regression function}.{\BBCQ}
\newblock
\APACjournalVolNumPages{Journal of the Royal Statistical Society: Series B
  (Methodological)}{39}{1}{107--113}.
\PrintBackRefs{\CurrentBib}

\bibitem [\protect \citeauthoryear {%
De~Vito%
, Massera%
, Piga%
, Martinotto%
\BCBL {}\ \BBA {} Di~Francia%
}{%
De~Vito%
\ \protect \BOthers {.}}{%
{\protect \APACyear {2008}}%
}]{%
de2008field}
\APACinsertmetastar {%
de2008field}%
\begin{APACrefauthors}%
De~Vito, S.%
, Massera, E.%
, Piga, M.%
, Martinotto, L.%
\BCBL {}\ \BBA {} Di~Francia, G.%
\end{APACrefauthors}%
\unskip\
\newblock
\APACrefYearMonthDay{2008}{}{}.
\newblock
{\BBOQ}\APACrefatitle {On field calibration of an electronic nose for benzene
  estimation in an urban pollution monitoring scenario} {On field calibration
  of an electronic nose for benzene estimation in an urban pollution monitoring
  scenario}.{\BBCQ}
\newblock
\APACjournalVolNumPages{Sensors and Actuators B: Chemical}{129}{2}{750--757}.
\PrintBackRefs{\CurrentBib}

\bibitem [\protect \citeauthoryear {%
Devroye%
}{%
Devroye%
}{%
{\protect \APACyear {1981}}%
}]{%
devroye1981laws}
\APACinsertmetastar {%
devroye1981laws}%
\begin{APACrefauthors}%
Devroye, L.%
\end{APACrefauthors}%
\unskip\
\newblock
\APACrefYearMonthDay{1981}{}{}.
\newblock
{\BBOQ}\APACrefatitle {Laws of the iterated logarithm for order statistics of
  uniform spacings} {Laws of the iterated logarithm for order statistics of
  uniform spacings}.{\BBCQ}
\newblock
\APACjournalVolNumPages{The Annals of Probability}{9}{5}{860--867}.
\PrintBackRefs{\CurrentBib}

\bibitem [\protect \citeauthoryear {%
Dhar%
}{%
Dhar%
}{%
{\protect \APACyear {2016}}%
}]{%
dhar2016trimmed}
\APACinsertmetastar {%
dhar2016trimmed}%
\begin{APACrefauthors}%
Dhar, S\BPBI S.%
\end{APACrefauthors}%
\unskip\
\newblock
\APACrefYearMonthDay{2016}{}{}.
\newblock
{\BBOQ}\APACrefatitle {Trimmed mean isotonic regression} {Trimmed mean isotonic
  regression}.{\BBCQ}
\newblock
\APACjournalVolNumPages{Scandinavian Journal of Statistics}{43}{1}{202--212}.
\PrintBackRefs{\CurrentBib}

\bibitem [\protect \citeauthoryear {%
Dhar%
\ \BBA {} Chaudhuri%
}{%
Dhar%
\ \BBA {} Chaudhuri%
}{%
{\protect \APACyear {2009}}%
}]{%
dhar2009comparison}
\APACinsertmetastar {%
dhar2009comparison}%
\begin{APACrefauthors}%
Dhar, S\BPBI S.%
\BCBT {}\ \BBA {} Chaudhuri, P.%
\end{APACrefauthors}%
\unskip\
\newblock
\APACrefYearMonthDay{2009}{}{}.
\newblock
{\BBOQ}\APACrefatitle {A comparison of robust estimators based on two types of
  trimming} {A comparison of robust estimators based on two types of
  trimming}.{\BBCQ}
\newblock
\APACjournalVolNumPages{Advances in Statistical Analysis}{93}{2}{151--158}.
\PrintBackRefs{\CurrentBib}

\bibitem [\protect \citeauthoryear {%
Dhar%
\ \BBA {} Chaudhuri%
}{%
Dhar%
\ \BBA {} Chaudhuri%
}{%
{\protect \APACyear {2012}}%
}]{%
dhar2012derivatives}
\APACinsertmetastar {%
dhar2012derivatives}%
\begin{APACrefauthors}%
Dhar, S\BPBI S.%
\BCBT {}\ \BBA {} Chaudhuri, P.%
\end{APACrefauthors}%
\unskip\
\newblock
\APACrefYearMonthDay{2012}{}{}.
\newblock
{\BBOQ}\APACrefatitle {On the derivatives of the trimmed mean} {On the
  derivatives of the trimmed mean}.{\BBCQ}
\newblock
\APACjournalVolNumPages{Statistica Sinica}{22}{2}{655--679}.
\PrintBackRefs{\CurrentBib}

\bibitem [\protect \citeauthoryear {%
Fan%
\ \BBA {} Gijbels%
}{%
Fan%
\ \BBA {} Gijbels%
}{%
{\protect \APACyear {1996}}%
}]{%
fan1996local}
\APACinsertmetastar {%
fan1996local}%
\begin{APACrefauthors}%
Fan, J.%
\BCBT {}\ \BBA {} Gijbels, I.%
\end{APACrefauthors}%
\unskip\
\newblock
\APACrefYear{1996}.
\newblock
\APACrefbtitle {Local Polynomial Modelling and Its Applications: Monographs on
  Statistics and Applied Probability 66} {Local polynomial modelling and its
  applications: Monographs on statistics and applied probability 66}\
  (\BVOL~66).
\newblock
\APACaddressPublisher{}{Chapman and Hall; London}.
\PrintBackRefs{\CurrentBib}

\bibitem [\protect \citeauthoryear {%
Gasser%
\ \BBA {} M{\"u}ller%
}{%
Gasser%
\ \BBA {} M{\"u}ller%
}{%
{\protect \APACyear {1979}}%
}]{%
gasser1979kernel}
\APACinsertmetastar {%
gasser1979kernel}%
\begin{APACrefauthors}%
Gasser, T.%
\BCBT {}\ \BBA {} M{\"u}ller, H\BHBI G.%
\end{APACrefauthors}%
\unskip\
\newblock
\APACrefYearMonthDay{1979}{}{}.
\newblock
{\BBOQ}\APACrefatitle {Kernel estimation of regression functions} {Kernel
  estimation of regression functions}.{\BBCQ}
\newblock
\BIn{} \APACrefbtitle {Smoothing techniques for curve estimation} {Smoothing
  techniques for curve estimation}\ (\BPGS\ 23--68).
\newblock
\APACaddressPublisher{}{Springer}.
\PrintBackRefs{\CurrentBib}

\bibitem [\protect \citeauthoryear {%
Hogg%
}{%
Hogg%
}{%
{\protect \APACyear {1967}}%
}]{%
hogg1967some}
\APACinsertmetastar {%
hogg1967some}%
\begin{APACrefauthors}%
Hogg, R\BPBI V.%
\end{APACrefauthors}%
\unskip\
\newblock
\APACrefYearMonthDay{1967}{}{}.
\newblock
{\BBOQ}\APACrefatitle {Some observations on robust estimation} {Some
  observations on robust estimation}.{\BBCQ}
\newblock
\APACjournalVolNumPages{Journal of the American Statistical
  Association}{62}{320}{1179--1186}.
\PrintBackRefs{\CurrentBib}

\bibitem [\protect \citeauthoryear {%
Huber%
}{%
Huber%
}{%
{\protect \APACyear {1981}}%
}]{%
huber1981robust}
\APACinsertmetastar {%
huber1981robust}%
\begin{APACrefauthors}%
Huber, P\BPBI J.%
\end{APACrefauthors}%
\unskip\
\newblock
\APACrefYearMonthDay{1981}{}{}.
\newblock
{\BBOQ}\APACrefatitle {Robust statistics} {Robust statistics}.{\BBCQ}
\newblock
\APACjournalVolNumPages{Wiley Series in Probability and Mathematical
  Statistics, New York: Wiley,| c1981}{}{}{}.
\PrintBackRefs{\CurrentBib}

\bibitem [\protect \citeauthoryear {%
Jaeckel%
}{%
Jaeckel%
}{%
{\protect \APACyear {1971}}%
}]{%
jaeckel1971some}
\APACinsertmetastar {%
jaeckel1971some}%
\begin{APACrefauthors}%
Jaeckel, L\BPBI A.%
\end{APACrefauthors}%
\unskip\
\newblock
\APACrefYearMonthDay{1971}{}{}.
\newblock
{\BBOQ}\APACrefatitle {Some flexible estimates of location} {Some flexible
  estimates of location}.{\BBCQ}
\newblock
\APACjournalVolNumPages{The Annals of Mathematical
  Statistics}{42}{5}{1540--1552}.
\PrintBackRefs{\CurrentBib}

\bibitem [\protect \citeauthoryear {%
Jure{\v{c}}kov{\'a}%
, Koenker%
\BCBL {}\ \BBA {} Welsh%
}{%
Jure{\v{c}}kov{\'a}%
\ \protect \BOthers {.}}{%
{\protect \APACyear {1994}}%
}]{%
jurevckova1994adaptive}
\APACinsertmetastar {%
jurevckova1994adaptive}%
\begin{APACrefauthors}%
Jure{\v{c}}kov{\'a}, J.%
, Koenker, R.%
\BCBL {}\ \BBA {} Welsh, A.%
\end{APACrefauthors}%
\unskip\
\newblock
\APACrefYearMonthDay{1994}{}{}.
\newblock
{\BBOQ}\APACrefatitle {Adaptive choice of trimming proportions} {Adaptive
  choice of trimming proportions}.{\BBCQ}
\newblock
\APACjournalVolNumPages{Annals of the Institute of Statistical
  Mathematics}{46}{4}{737--755}.
\PrintBackRefs{\CurrentBib}

\bibitem [\protect \citeauthoryear {%
Jureckov{\'a}%
\ \BBA {} Proch{\'a}zka%
}{%
Jureckov{\'a}%
\ \BBA {} Proch{\'a}zka%
}{%
{\protect \APACyear {1994}}%
}]{%
jureckova1994regression}
\APACinsertmetastar {%
jureckova1994regression}%
\begin{APACrefauthors}%
Jureckov{\'a}, J.%
\BCBT {}\ \BBA {} Proch{\'a}zka, B.%
\end{APACrefauthors}%
\unskip\
\newblock
\APACrefYearMonthDay{1994}{}{}.
\newblock
{\BBOQ}\APACrefatitle {Regression quantiles and trimmed least squares estimator
  in nonlinear regression model} {Regression quantiles and trimmed least
  squares estimator in nonlinear regression model}.{\BBCQ}
\newblock
\APACjournalVolNumPages{Journal of Nonparametric Statistics}{3}{3}{201--222}.
\PrintBackRefs{\CurrentBib}

\bibitem [\protect \citeauthoryear {%
Kaya%
, T{\"u}fekci%
\BCBL {}\ \BBA {} G{\"u}rgen%
}{%
Kaya%
\ \protect \BOthers {.}}{%
{\protect \APACyear {2012}}%
}]{%
kaya2012local}
\APACinsertmetastar {%
kaya2012local}%
\begin{APACrefauthors}%
Kaya, H.%
, T{\"u}fekci, P.%
\BCBL {}\ \BBA {} G{\"u}rgen, F\BPBI S.%
\end{APACrefauthors}%
\unskip\
\newblock
\APACrefYearMonthDay{2012}{}{}.
\newblock
{\BBOQ}\APACrefatitle {Local and global learning methods for predicting power
  of a combined gas \& steam turbine} {Local and global learning methods for
  predicting power of a combined gas \& steam turbine}.{\BBCQ}
\newblock
\BIn{} \APACrefbtitle {Proceedings of the International Conference on Emerging
  Trends in Computer and Electronics Engineering ICETCEE} {Proceedings of the
  international conference on emerging trends in computer and electronics
  engineering icetcee}\ (\BPGS\ 13--18).
\PrintBackRefs{\CurrentBib}

\bibitem [\protect \citeauthoryear {%
Nadaraya%
}{%
Nadaraya%
}{%
{\protect \APACyear {1965}}%
}]{%
nadaraya1965non}
\APACinsertmetastar {%
nadaraya1965non}%
\begin{APACrefauthors}%
Nadaraya, E.%
\end{APACrefauthors}%
\unskip\
\newblock
\APACrefYearMonthDay{1965}{}{}.
\newblock
{\BBOQ}\APACrefatitle {On non-parametric estimates of density functions and
  regression curves} {On non-parametric estimates of density functions and
  regression curves}.{\BBCQ}
\newblock
\APACjournalVolNumPages{Theory of Probability \& Its
  Applications}{10}{1}{186--190}.
\PrintBackRefs{\CurrentBib}

\bibitem [\protect \citeauthoryear {%
Park%
, Lee%
\BCBL {}\ \BBA {} Chang%
}{%
Park%
\ \protect \BOthers {.}}{%
{\protect \APACyear {2015}}%
}]{%
park2015robust}
\APACinsertmetastar {%
park2015robust}%
\begin{APACrefauthors}%
Park, C\BHBI H.%
, Lee, S.%
\BCBL {}\ \BBA {} Chang, J\BHBI H.%
\end{APACrefauthors}%
\unskip\
\newblock
\APACrefYearMonthDay{2015}{}{}.
\newblock
{\BBOQ}\APACrefatitle {Robust closed-form time-of-arrival source localization
  based on $\alpha$-trimmed mean and Hodges--Lehmann estimator under NLOS
  environments} {Robust closed-form time-of-arrival source localization based
  on $\alpha$-trimmed mean and hodges--lehmann estimator under nlos
  environments}.{\BBCQ}
\newblock
\APACjournalVolNumPages{Signal Processing}{111}{}{113--123}.
\PrintBackRefs{\CurrentBib}

\bibitem [\protect \citeauthoryear {%
Priestley%
\ \BBA {} Chao%
}{%
Priestley%
\ \BBA {} Chao%
}{%
{\protect \APACyear {1972}}%
}]{%
priestley1972non}
\APACinsertmetastar {%
priestley1972non}%
\begin{APACrefauthors}%
Priestley, M\BPBI B.%
\BCBT {}\ \BBA {} Chao, M.%
\end{APACrefauthors}%
\unskip\
\newblock
\APACrefYearMonthDay{1972}{}{}.
\newblock
{\BBOQ}\APACrefatitle {Non-parametric function fitting} {Non-parametric
  function fitting}.{\BBCQ}
\newblock
\APACjournalVolNumPages{Journal of the Royal Statistical Society: Series B
  (Methodological)}{34}{3}{385--392}.
\PrintBackRefs{\CurrentBib}

\bibitem [\protect \citeauthoryear {%
Rousseeuw%
\ \BBA {} Leroy%
}{%
Rousseeuw%
\ \BBA {} Leroy%
}{%
{\protect \APACyear {2005}}%
}]{%
rousseeuw2005robust}
\APACinsertmetastar {%
rousseeuw2005robust}%
\begin{APACrefauthors}%
Rousseeuw, P\BPBI J.%
\BCBT {}\ \BBA {} Leroy, A\BPBI M.%
\end{APACrefauthors}%
\unskip\
\newblock
\APACrefYear{2005}.
\newblock
\APACrefbtitle {Robust regression and outlier detection} {Robust regression and
  outlier detection}\ (\BVOL~589).
\newblock
\APACaddressPublisher{}{John wiley \& sons}.
\PrintBackRefs{\CurrentBib}

\bibitem [\protect \citeauthoryear {%
Serfling%
}{%
Serfling%
}{%
{\protect \APACyear {2009}}%
}]{%
serfling2009approximation}
\APACinsertmetastar {%
serfling2009approximation}%
\begin{APACrefauthors}%
Serfling, R\BPBI J.%
\end{APACrefauthors}%
\unskip\
\newblock
\APACrefYear{2009}.
\newblock
\APACrefbtitle {Approximation theorems of mathematical statistics}
  {Approximation theorems of mathematical statistics}\ (\BVOL~162).
\newblock
\APACaddressPublisher{}{John Wiley \& Sons}.
\PrintBackRefs{\CurrentBib}

\bibitem [\protect \citeauthoryear {%
Silverman%
}{%
Silverman%
}{%
{\protect \APACyear {1986}}%
}]{%
silverman1986density}
\APACinsertmetastar {%
silverman1986density}%
\begin{APACrefauthors}%
Silverman, B\BPBI W.%
\end{APACrefauthors}%
\unskip\
\newblock
\APACrefYear{1986}.
\newblock
\APACrefbtitle {Density Estimation for Statistics and Data Analysis} {Density
  estimation for statistics and data analysis}\ (\BVOL~26).
\newblock
\APACaddressPublisher{}{CRC Press}.
\PrintBackRefs{\CurrentBib}

\bibitem [\protect \citeauthoryear {%
Stigler%
}{%
Stigler%
}{%
{\protect \APACyear {1973}}%
}]{%
stigler1973asymptotic}
\APACinsertmetastar {%
stigler1973asymptotic}%
\begin{APACrefauthors}%
Stigler, S\BPBI M.%
\end{APACrefauthors}%
\unskip\
\newblock
\APACrefYearMonthDay{1973}{}{}.
\newblock
{\BBOQ}\APACrefatitle {The asymptotic distribution of the trimmed mean} {The
  asymptotic distribution of the trimmed mean}.{\BBCQ}
\newblock
\APACjournalVolNumPages{The Annals of Statistics}{1}{3}{472--477}.
\PrintBackRefs{\CurrentBib}

\bibitem [\protect \citeauthoryear {%
Tsanas%
, Little%
, McSharry%
\BCBL {}\ \BBA {} Ramig%
}{%
Tsanas%
\ \protect \BOthers {.}}{%
{\protect \APACyear {2009}}%
}]{%
tsanas2009accurate}
\APACinsertmetastar {%
tsanas2009accurate}%
\begin{APACrefauthors}%
Tsanas, A.%
, Little, M\BPBI A.%
, McSharry, P\BPBI E.%
\BCBL {}\ \BBA {} Ramig, L\BPBI O.%
\end{APACrefauthors}%
\unskip\
\newblock
\APACrefYearMonthDay{2009}{}{}.
\newblock
{\BBOQ}\APACrefatitle {Accurate telemonitoring of Parkinson's disease
  progression by noninvasive speech tests} {Accurate telemonitoring of
  parkinson's disease progression by noninvasive speech tests}.{\BBCQ}
\newblock
\APACjournalVolNumPages{IEEE transactions on Biomedical
  Engineering}{57}{4}{884--893}.
\PrintBackRefs{\CurrentBib}

\bibitem [\protect \citeauthoryear {%
T{\"u}fekci%
}{%
T{\"u}fekci%
}{%
{\protect \APACyear {2014}}%
}]{%
tufekci2014prediction}
\APACinsertmetastar {%
tufekci2014prediction}%
\begin{APACrefauthors}%
T{\"u}fekci, P.%
\end{APACrefauthors}%
\unskip\
\newblock
\APACrefYearMonthDay{2014}{}{}.
\newblock
{\BBOQ}\APACrefatitle {Prediction of full load electrical power output of a
  base load operated combined cycle power plant using machine learning methods}
  {Prediction of full load electrical power output of a base load operated
  combined cycle power plant using machine learning methods}.{\BBCQ}
\newblock
\APACjournalVolNumPages{International Journal of Electrical Power \& Energy
  Systems}{60}{}{126--140}.
\PrintBackRefs{\CurrentBib}

\bibitem [\protect \citeauthoryear {%
Wang%
, Lin%
\BCBL {}\ \BBA {} Tang%
}{%
Wang%
\ \protect \BOthers {.}}{%
{\protect \APACyear {2019}}%
}]{%
wang2019robust}
\APACinsertmetastar {%
wang2019robust}%
\begin{APACrefauthors}%
Wang, W.%
, Lin, N.%
\BCBL {}\ \BBA {} Tang, X.%
\end{APACrefauthors}%
\unskip\
\newblock
\APACrefYearMonthDay{2019}{}{}.
\newblock
{\BBOQ}\APACrefatitle {Robust two-sample test of high-dimensional mean vectors
  under dependence} {Robust two-sample test of high-dimensional mean vectors
  under dependence}.{\BBCQ}
\newblock
\APACjournalVolNumPages{Journal of Multivariate Analysis}{169}{}{312--329}.
\PrintBackRefs{\CurrentBib}

\bibitem [\protect \citeauthoryear {%
Watson%
}{%
Watson%
}{%
{\protect \APACyear {1964}}%
}]{%
watson1964smooth}
\APACinsertmetastar {%
watson1964smooth}%
\begin{APACrefauthors}%
Watson, G\BPBI S.%
\end{APACrefauthors}%
\unskip\
\newblock
\APACrefYearMonthDay{1964}{}{}.
\newblock
{\BBOQ}\APACrefatitle {Smooth regression analysis} {Smooth regression
  analysis}.{\BBCQ}
\newblock
\APACjournalVolNumPages{Sankhy{\=a}: The Indian Journal of Statistics, Series
  A}{26}{4}{359--372}.
\PrintBackRefs{\CurrentBib}

\bibitem [\protect \citeauthoryear {%
Welsh%
}{%
Welsh%
}{%
{\protect \APACyear {1987}}%
}]{%
welsh1987trimmed}
\APACinsertmetastar {%
welsh1987trimmed}%
\begin{APACrefauthors}%
Welsh, A.%
\end{APACrefauthors}%
\unskip\
\newblock
\APACrefYearMonthDay{1987}{}{}.
\newblock
{\BBOQ}\APACrefatitle {The trimmed mean in the linear model} {The trimmed mean
  in the linear model}.{\BBCQ}
\newblock
\APACjournalVolNumPages{The Annals of Statistics}{15}{1}{20--36}.
\PrintBackRefs{\CurrentBib}

\end{thebibliography}

\end{document}